\documentclass{amsart}

\usepackage{amssymb, amsmath,mathtools,mathabx}
\usepackage{mathrsfs}%stix}
\usepackage{amscd}
\usepackage{verbatim}
\usepackage{stmaryrd}
\usepackage[dvipsnames]{xcolor}
\usepackage{a4wide}
\usepackage{tikz}
\usetikzlibrary{shapes.geometric, arrows}
\usepackage{stmaryrd}

\usepackage{enumerate}

\mathtoolsset{showonlyrefs}
\numberwithin{equation}{section}
\usepackage{nicefrac}

\usepackage[colorlinks,linkcolor={blue},citecolor={blue},urlcolor={purple},]{hyperref}

\usepackage[colorinlistoftodos,prependcaption,textsize=tiny]{todonotes}

\usepackage{forest}
\usepackage{tikz}
\tikzset{>=latex}

\usepackage{tabularx}
\newcolumntype{L}{>{\arraybackslash}X}
\usepackage{multirow}

\theoremstyle{plain}
\newtheorem{theorem}{Theorem}[section]
\theoremstyle{remark}
\newtheorem{remark}[theorem]{Remark}

\theoremstyle{plain}
\newtheorem{corollary}[theorem]{Corollary}
\newtheorem{lemma}[theorem]{Lemma}
\newtheorem{proposition}[theorem]{Proposition}
\newtheorem{definition}[theorem]{Definition}

\newtheorem{assumption}[theorem]{Assumption}

\numberwithin{equation}{section}

\newcommand{\N}{\mathbb{N}}

\newcommand{\R}{\mathbb{R}}

\newcommand{\ceqq}{\coloneqq}
\newcommand{\eqqc}{\eqqcolon}
\newcommand{\col}{\colon}  
\newcommand{\pf}{p_{\kern-.5pt \alpha}}

\newcommand{\wk}{w_{\kappa}}
\newcommand{\wn}{w_{\nu}}
\newcommand{\wni}{w_{\nu_i}}

\newcommand{\calL}{\mathcal{L}}

\makeatletter
\newcommand{\plus}{%
  \DOTSB\mathop{\mathpalette\mattos@bigplus\relax}\slimits@
}
\newcommand\mattos@bigplus[2]{%
  \vcenter{\hbox{%
    \sbox\z@{$#1\sum$}%
    \resizebox{!}{0.9\dimexpr\ht\z@+\dp\z@}{\raisebox{\depth}{$\m@th#1+$}}%
  }}%
  \vphantom{\sum}%
}
\makeatother

\newcommand{\x}{{L_+}}

\newcommand{\Xtr}{X_{1-\frac{1+\kappa}{p},p}}
\newcommand{\Xtrzero}{X_{1-\frac{1}{p},p}}

\newcommand{\dd}{\,\mathrm{d}}

\newcommand{\wt}{\widetilde}

\renewcommand{\MR}{\mathrm{MR}}

\newcommand{\one}{{{\bf 1}}}

\newcommand{\lb}{\langle}
\newcommand{\rb}{\rangle}

\newcommand{\norm}[1]{{\left\vert\kern-0.25ex\left\vert\kern-0.25ex\left\vert #1
    \right\vert\kern-0.25ex\right\vert\kern-0.25ex\right\vert}}

\newcommand{\into}{\hookrightarrow}

\def\XXint#1#2#3{{\setbox0=\hbox{$#1{#2#3}{\int}$ }
\vcenter{\hbox{$#2#3$ }}\kern-.6\wd0}}

\newcommand{\zz}{z}

\allowdisplaybreaks

\begin{document}

\title[Maximal regularity with critical singular perturbations]{Maximal regularity for evolution equations with critical singular perturbations}


\author{Esm\'ee Theewis}
\address[Esm\'ee Theewis]{Delft Institute of Applied Mathematics\\
Delft University of Technology \\ P.O. Box 5031\\ 2600 GA Delft\\The
Netherlands}
\email{e.s.theewis@tudelft.nl}
\author{Mark Veraar}
\address[Mark Veraar]{Delft Institute of Applied Mathematics\\
Delft University of Technology \\ P.O. Box 5031\\ 2600 GA Delft\\The
Netherlands.} \email{m.c.veraar@tudelft.nl} 

\thanks{The authors are supported by the VICI subsidy VI.C.212.027 of the Netherlands Organisation for Scientific Research (NWO)}

\begin{abstract}
Assuming $A$ has maximal $L^p$-regularity, this paper investigates perturbations of $A$ by time-dependent operators $B$ that are unbounded and satisfy a critical $L^q$-integrability condition in time. We establish two main results. The first proves maximal $L^p$-regularity for the critical endpoint case, generalizing previous work by Pr\"uss and Schnaubelt (2001).  The second develops a weighted maximal regularity theory for mixed-scale perturbations, motivated by the linearized skeleton equations appearing in large deviations theory for stochastic PDEs. 
\end{abstract}



\maketitle

\section{Introduction}

The theory of maximal $L^p$-regularity has for many years played an important role in the study of nonlinear evolution equations. For details, the reader is referred to  \cite{Am, DHP, HNVWvolume3, pruss2016moving} and the  surveys \cite{Denkintro, Monn24, Wilke23}. 
Recently, there have also been significant developments in the setting of stochastic evolution equations. An overview can be found in  \cite{AV25survey}. 

While this paper focuses exclusively on deterministic evolution equations, our primary motivation stems from large deviations theory for stochastic evolution equations. 
Establishing a large deviation principle requires the analysis of a nonstandard perturbed equation known as the skeleton equation (discussed below). A main objective of this paper is to prove maximal $L^p$-regularity for its linearized variant -- a problem that has previously been understood only in special cases.  

For $p=2$, within a variational setting, 
this was considered in our recent paper \cite[Cor.\ 3.5]{TV24} (see also the references therein). A similar problem was studied in \cite[\S4]{BGV24} to allow for certain singular coefficients in a variational setting. 
The linear part of the analysis in these papers treats the case where $A\col[0,T]\to \calL(X_1, X_0)$ is a {\em coercive} time-dependent operator in a variational framework. Here, we are given a Gelfand triple $(X_1, X_{\nicefrac12}, X_0)$ of Hilbert spaces, where $X_1\hookrightarrow X_0$ densely and continuously and $X_{\nicefrac12} = [X_0, X_1]_{\nicefrac12}$ is the complex interpolation space. From Lions' theory, it is well-known that $A$ has maximal $L^2$-regularity.

The linearized {\em skeleton equation} appearing in large deviations theory takes the form
\begin{align}\label{eq:intropert}
  u' + Au + Bu = f+ g,
\end{align} with $u(0) = u_0\in X_{\nicefrac12}$. Here, $B\in L^2(0,T;\calL(X_1, X_{\nicefrac12}))$, $f\in L^2(0,T;X_0)$ and $g\in L^1(0,T;X_{\nicefrac12})$. For details on the precise form of $B$, we refer to Subsection \ref{ss:intro2}. 
The perturbation $B$ is singular in time, as $\|B(\cdot)\|_{\calL(X_1, X_{\nicefrac12})}$ is allowed to be unbounded. Moreover, the perturbation is critical: the space-time Sobolev indices of the leading term $Au$ and the perturbation $Bu$ coincide since $-\frac12 + 0 = -\frac12$ and $-\frac{1}{1} + \frac12 = - \frac12$, respectively. \footnote{In general the space-time Sobolev index of $v\in  L^p(0,T;X_{\theta})$ is $-\frac{1}{p} + \theta$.} 
Thus, $B$ is a borderline perturbation regarding the space-time Sobolev index. 
Consequently, it is not obvious that the perturbed operator $A+B$ preserves maximal $L^2$-regularity.

In \cite[\S3]{TV24},  we developed new techniques to establish maximal $L^2$-regularity of $A+B$, relying on the coercivity of $A$. In the current paper, we also cover the {\em non-coercive} case in Section \ref{sec:borderline}, thus enabling  applications to   large deviations in $L^2$-settings where the leading operator is not necessarily coercive.  

To develop an $L^p$-theory of large deviations with $p\in (2, \infty)$, we require an $L^p$-variant of the aforementioned results. In this setting, the endpoint $L^1$ case is no longer required, allowing us to use a different approach compared to previous work. 
Indeed,  replacing $L^2$ by $L^p$ with $p\in(2,\infty)$ in the above setting, the natural condition on $B$ turns out to be $B\in L^2(0,T;\calL(X_1,X_{\nicefrac12}))$, and thus $Bu\in L^r(0,T;X_{\nicefrac12})$ with $\frac{1}{p}+\frac12 = \frac1r$ and $r\in (1, 2)$. Here, the space-time Sobolev indices of $Au$ and $Bu$ still coincide (both are $-\frac1p$), showing that $B$ is in this sense a critical perturbation of $A$.

While the literature on perturbation results is vast, most works are restricted to the autonomous setting. For time-dependent and singular perturbations, results on maximal $L^p$-regularity are scarce. The works most related to our setting include \cite{ABDH, prussschnaubelt01}.

In this paper, we obtain two perturbation theorems for maximal $L^p$-regularity in an abstract evolution equation setting: 
\begin{enumerate}[\rm (1)]
\item A generalization of the perturbation result on maximal $L^p$-regularity obtained in \cite[Th.\ 3.1]{prussschnaubelt01}. 
\item A maximal $L^p$-regularity theory for linear skeleton equations as they appear in large deviations theory, and generalizations thereof.  
\end{enumerate}
    The starting point will always be a time-dependent operator family $A$ which has maximal $L^p$-regularity on a Banach space $X_0$. The operator $B$ will be a lower order  
    perturbation which is singular in time, and which satisfies a borderline integrability condition. 
While our two perturbation results are related, their strongest versions are obtained by treating the settings separately. This separation also benefits the exposition, as the first result involves fewer technical complexities.  

In the next two subsections, we state special cases of the main results. 
\subsection{Setting with full maximal $L^p$-regularity}
Suppose that the following assumption holds:
\begin{assumption}\label{assump:XA}
Let $T\in (0,\infty)$. 
Let $X_0$ and $X_1$ be Banach spaces such that $X_1\hookrightarrow X_0$ densely and continuously. 
Suppose that $A\col[0,T]\to \calL(X_1, X_0)$ is strongly measurable in the strong operator topology. 
\end{assumption}
Note that we allow   $\|A(\cdot)\|_{\calL(X_1, X_0)}$ to be unbounded in time. 
Given that $A$ has maximal $L^p$-regularity for a given $p\in (1, \infty)$ (see Definition \ref{def:MR}), we want to find conditions under which $A+B$ has maximal $L^p$-regularity as well, where $B$ is a singular lower order perturbation. 

\begin{theorem}\label{thm:MRI}
Suppose that Assumption \ref{assump:XA} holds. Let $p\in (1, \infty)$ and suppose that there exists a constant $C_A$ such that 
\[
\|Av\|_{L^p(0,T;X_0)}\leq C_A ( \|v\|_{L^p(0,T;X_1)}+\|v\|_{W^{1,p}(0,T;X_0)}), \quad v\in L^p(0,T;X_1)\cap W^{1,p}(0,T;X_0). 
\]
Suppose that $A$ has maximal $L^p$-regularity. Let  $q\in (p, \infty)$ and let $B\col[0,T]\to \calL(X_{1-\nicefrac{1}{q},1},X_0)$ be strongly measurable in the strong operator topology. Suppose that \[\|B(t)\|_{\calL(X_{1-\nicefrac{1}{q},1},X_0)}\leq b(t), \qquad t\in [0,T],\]
where $b\in L^q(0,T)$. Then $A+B$ has maximal $L^p$-regularity.
\end{theorem}
The above theorem is a special case of Theorem \ref{thm:perturbLp}. 
The main difficulty in the proof of the above is that $B$ is a singular perturbation which is also critical in the sense that the $L^q$-integrability condition on $b$ is optimal. 
Indeed, by trace theory, the maximal regularity space $L^p(0,T;X_1)\cap W^{1,p}(0,T;X_0)$ embeds into $L^{s}(0,T;X_{1-\nicefrac{1}{q},1})$ with $\frac1p = \frac{1}{q} + \frac{1}{s}$ (see \cite[Th.\ L.4.4]{HNVWvolume3}).  
Combined with H\"older's inequality, we see that  
\[\|B u \|_{L^p(0,T;X_0)}\leq \|b\|_{L^q(0,T)} \|u\|_{L^{s}(0,T;X_{1-\nicefrac{1}{q},1})}\lesssim \|b\|_{L^q(0,T)} \|u\|_{L^p(0,T;X_1)\cap W^{1,p}(0,T;X_0)},\]
whenever $u$ belongs to the maximal regularity space. This shows that the integrability condition on $b$ is natural, but also critical. Note that $X_{1-\nicefrac{1}{q},1}$ can be replaced by the complex interpolation space $[X_0, X_1]_{1-\nicefrac1q}$ or any intermediate space $X_{1-\nicefrac1q}$ which is of class $J_{1-\nicefrac{1}{q}}$ between $X_0$ and $X_1$  (see \cite[Chap.\ 1]{Lun}). Indeed, this follows from $X_{1-\nicefrac{1}{q},1}\hookrightarrow X_{1-\nicefrac1q}$. 

The limiting case $q=p$ can also be included if the space  $X_{1-\nicefrac{1}{q},1}$ is replaced by the trace space $X_{1-\frac{1}{p},p}$. The latter case is easier to prove and is  presented in Theorem \ref{thm:perturbGronwall}.

Our findings generalize the perturbation result of \cite[Th.\ 3.1]{prussschnaubelt01}. Therein, maximal $L^p$-regularity of $A+B$ is also established under the condition  $\|B(\cdot)\|_{\calL(X_{1-\nicefrac{1}{q},1},X_0)}\leq b(\cdot)$, but it is required that $b \in L^r(0,T)$ with $r > q$. We address the critical endpoint $r=q$.  
Other differences between our result and \cite[Th.\ 3.1]{prussschnaubelt01} are 
that we do not assume continuity of $A$ and we include temporal weights to cover rough initial data.

\subsection{Setting beyond maximal $L^p$-regularity}\label{ss:intro2}

In this section, we state a special case of our second main result. This setting is designed specifically to cover the linear maximal regularity theory required for our follow-up work on nonlinear skeleton equations and their applications to large deviations.  
In the introduction, for simplicity, we only consider the case where $A$ is time-independent, and refer for the case of time-dependent $A$ to Theorem \ref{th: well-posed linear}. 

In the weak convergence approach \cite{BudhiDup} to large deviations for SPDE (see \cite{HLL,TV24} and references therein), a multiplicative stochastic term $\mathcal{B}(t) u(t) \mathrm{d} W(t)$, where $\mathcal{B}(t)\col X_1\to X_{\nicefrac12}$ is replaced by $\mathcal{B}(t)u(t) \phi(t) \mathrm{d} t$ with $\mathcal{B}\in L^\infty(0,T;\calL(X_1, X_{\nicefrac12}))$ and $\phi\in L^2(0,T)$. Here $W$ could be a standard Brownian motion, or (with minor modifications to $\mathcal{B}$ and $\phi$),  a cylindrical Brownian motion. It turns out that we need to understand maximal regularity for the problem
\begin{align}\label{eq:intropertLDP}
u'(t) + A u(t)+ B(t) u(t) = f(t)+g(t),
\end{align}
with $u(0) = u_0$, and where $B(t) \ceqq \mathcal{B}(t)\phi(t)$. Nonlinear terms can be included afterwards by linearization techniques. Thus, we first focus on the leading order linear terms. 
Typically, one has $B\in L^2(0,T;\calL(X_1, X_{\nicefrac12}))$, where $X_{\nicefrac12}$ could be the real interpolation space $(X_0, X_1)_{\nicefrac12,q}$ (with $q\in [1, \infty]$) or the complex interpolation space $[X_0, X_1]_{\nicefrac12}$. As mentioned before in \eqref{eq:intropert}, such a $B$ is a critical perturbation. 

Now if $A$ has maximal $L^p$-regularity, then we need to prove that for $f\in L^p(0,T;X_0)$ and $g\in L^{r}(0,T;X_{\gamma})$ with $\frac{1}{r} = \frac1p+\frac12$, the problem above has a solution $u\in L^p(0,T;X_1)$. For stochastic problems, we need $p\in [2,\infty)$, so that $r\in [1, 2)$. The case $p=2$, $r=1$ is an exceptional endpoint, which was already  mentioned in \eqref{eq:intropert}, but we return to it in Section \ref{sec:borderline}. Below, we focus on $p\in (2, \infty)$ for convenience. 

There is a connection to Theorem \ref{thm:MRI} in the case $q=2$. However, a difference is that the solution to \eqref{eq:intropertLDP} cannot be in $W^{1,p}(0,T;X_0)$. Indeed, $B u$ and $g$ are not assumed to be in $L^p(0,T;X_0)$, but only in $L^{r}(0,T;X_{\gamma})$. Thus, the time regularity has to be described in a different way. The best strategy seems to be the use of a sum of maximal regularity spaces for $u$: 
\begin{equation}\label{eq:MRspaceintro}
  \Big(L^p(0,T;X_1)\cap W^{1,p}(0,T;X_0)\Big)+\Big(L^r(0,T;X_{\nicefrac32})\cap W^{1,r}(0,T;X_{\nicefrac12})\Big). 
\end{equation}
Indeed, since $Bu$ and $g$ are in $L^r(0,T;X_{\nicefrac12})$, the above appears to be the optimal space for $u$ if we apply both  maximal $L^p$-regularity and maximal $L^r$-regularity to the respective  parts of the equation. Note that the space-time Sobolev indices of the components in the sum above coincide: $1-\frac1p = \frac{3}{2} - \frac1r$. We prove that this sum maximal regularity space still embeds into
\[L^p(0,T;X_1)\cap C([0,T];X_{1-\nicefrac1p,p}),
\]
which is a typical regularity class for stochastic evolution equations.  

We highlight here the following special case of  Theorem \ref{th: well-posed linear} and Corollary \ref{cor:time-indep}:
\begin{theorem}\label{thm:introhalfB}
Suppose that Assumption \ref{assump:XA} holds, and that $A$ is time-independent. Let $p\in (2, \infty)$ and suppose that $A$ has maximal $L^p$-regularity. Let $\frac{1}{r} = \frac1p+\frac12$ and suppose that $B\col[0,T]\to \calL(X_{1},X_{\nicefrac12})$ is strongly measurable in the strong operator topology, and suppose that \[\|B(t)\|_{\calL(X_{1},X_{\nicefrac12})}\leq b(t), \qquad t\in [0,T],\]
where $b\in L^2(0,T)$. Then for all $u_0\in X_{1-\nicefrac1p,p}$, $f\in L^p(0,T;X_0)$ and $g\in L^r(0,T;X_{\nicefrac12})$, there exists a unique strong solution $u$ to \eqref{eq:intropertLDP} in the space \eqref{eq:MRspaceintro}. Moreover, $u\in C([0,T];X_{1-\nicefrac1p,p})$.  
\end{theorem}

In the main text we also cover the more general case $B(t)\in \calL(X_{1},X_{\gamma})$, where again $X_{\gamma}$ could, for instance, be any complex or real interpolation space. 

Finally,  the limiting case $r=1$ is considered in Section \ref{sec:borderline}. 
In particular, we present a version of Theorem \ref{thm:introhalfB} with $p=2$, where $X_{\nicefrac12}$ is the real interpolation space $(X_0, X_1)_{\nicefrac12,2}$. 

Table \ref{tab:perturbation_lp} summarizes our perturbation results in a special case, without temporal weights. In the table, we denote $I=(0,T)$.   
\begin{table}[ht]
    \centering
    \caption{Maximal $L^p$-regularity of $A+B$} 

    \label{tab:perturbation_lp}
    \renewcommand{\arraystretch}{1.5} % Increases row height for readability
    \begin{tabular}{|l|l|l|l|l|}
        \hline
         & Perturbation $B$ 
         & $\|B(\cdot)\| \in L^q$ & Inhomogeneity $f$ & Solution $u$ \\ [2.5pt] 
        \hline
        Th.\ \ref{thm:perturbLp} & $X_{\frac{1}{q'}, 1} \to X_0$ & $q \in (p, \infty)$ & $L^p(I;X_0)$ & $L^p(I;X_1)\cap W^{1,p}(I;X_0)$ \\ [2.5pt] 
        \hline
        Th.\ \ref{thm:perturbGronwall} & $X_{\frac{1}{p'}, p} \to X_0$ & $q = p$ & $L^p(I;X_0)$ & $L^p(I;X_1)\cap W^{1,p}(I;X_0)$ \\ [2.5pt] 
        \hline
        Th.\ \ref{th: well-posed linear} & $X_1 \to X_{\frac{1}{q}, \infty}$ & $\frac{1}{q} = \frac1r-\frac1p$ & $L^p(I;X_0) + L^r(I;X_{\frac1q,\infty})$ & $L^p(I;X_1)\cap C(I;X_{\frac1{p'},p})$ \\ [2.5pt] 
        \hline
        Th.\ \ref{thm:mainHilbert} & $X_1 \to X_{\frac{1}{p'},p}$ & $\frac{1}{q} = 1-\frac1p$ & $L^p(I;X_0) + L^1(I;X_{\frac1{p'},p})$ & $L^p(I;X_1)\cap C(I;X_{\frac1{p'},p})$ \\ [2.5pt] 
        \hline
    \end{tabular}
\end{table}

\subsubsection*{Acknowledgements}
The authors thank Emiel Lorist for pointing out the simple proof of Lemma \ref{lem:L2real}.

\section{Perturbations of maximal $L^p$-regularity}\label{sec:PS extension}

In this section, we establish a new perturbation result for maximal $L^p$-regularity, which implies Theorem \ref{thm:MRI}. We begin by outlining the necessary preliminaries regarding maximal $L^p$-regularity for operator families $A$ that are singular in time.

\subsection{Maximal $L^p$-regularity}\label{sub: prelims max Lp reg}

Let $X_0$ and $X_1$ be Banach spaces such that $X_1\hookrightarrow X_0$ densely and continuously. 
Suppose that $A\col[0,T]\to \calL(X_1, X_0)$ is strongly measurable in the strong operator topology. 
In contrast to standard literature, we allow the map $t\mapsto \|A(t)\|_{\calL(X_1, X_0)}$ to be unbounded, a level of generality required for Theorem \ref{thm:MRI}.

For $f\in L^1(0,T;X_0)$, we consider the problem
\begin{equation}
\label{eq:parabolic_problem_prel}
\left\{
\begin{aligned}
    u' + A u &= f, \\
    u(0) &= u_0.
\end{aligned}
\right.
\end{equation}
A function $u\in L^1(0,T;X_1)\cap W^{1,1}(0,T;X_0)$  is called a {\em strong solution} to \eqref{eq:parabolic_problem_prel} if $A u\in L^1(0,T;X_0)$ and for all $t\in [0,T]$, 
\[u(t) = u_0-\int_0^t A(s) u(s) ds + \int_0^t f(s) ds.\]
Note that in this case $u\in C([0,T];X_0)$ and this is the version of $u$ that we will use from now on. Clearly, this definition can be extended to other intervals $(a,b)$. 

To relax the regularity conditions on the initial data, we employ power-weighted $L^p$-spaces. We adopt the following notation. For $p\in (1, \infty)$ and $\kappa\in [0,p-1)$, let $\wk(t)\ceqq |t|^\kappa$ and let $L^p(a,b,\wk;X_0)$ denote the space of all strongly measurable $f\col(a,b)\to X_{0}$ such that 
\begin{align*}
\|f\|_{L^p(a,b,\wk;X_0)}\ceqq \Big(\int_a^b t^\kappa \|f(t)\|_{X_0}^p\dd t\Big)^{1/p}<\infty
\end{align*} 
One can check that $L^p(a,b,\wk;X_0)\hookrightarrow L^1(a,b;X_0)$. The space $W^{1,p}(a,b,w_{\kappa};X_0)$ is defined in a similar way.

In the sequel, for $0\leq a<b\leq \infty$, we will use the shorthand notation 
\begin{equation}\label{eq:def MRp}
\MR^p(a,b,w_{\kappa}) \ceqq  L^p(a,b,w_{\kappa};X_1)\cap W^{1,p}(a,b,w_{\kappa};X_0).
\end{equation}
The above space will be equipped with the norm 
\[\|u\|_{\MR^p(a,b,w_{\kappa})} \ceqq  \max\{\|u\|_{L^p(a,b,w_{\kappa};X_1)},\|u\|_{W^{1,p}(a,b,w_{\kappa};X_0)}\},\]
where we emphasize that we do not shift the weight. When $\kappa=0$, we will write $\MR^p(a,b)$. Note that for $a>0$, one has $\MR^p(a,b,w_{\kappa}) = \MR^p(a,b)$ isomorphically since the weight is bounded from above and below. 

We only consider $\kappa\in [0, p-1)$. 
The lower bound on $\kappa$ is motivated by the goal of enlarging the space of initial values by including the power weight.  
The upper bound on $\kappa$ ensures that the functions  we consider are integrable on $[0,T]$.

Recall from \cite[Cor.\ L.4.6]{HNVWvolume3} that the following trace embedding results hold:
\begin{align}
\label{eq:traceembedding0T} \MR^p(0,T,w_{\kappa})& \hookrightarrow C([0,T];X_{1-\frac{1+\kappa}{p},p}),
\\
\label{eq:traceembeddingaT} \MR^p(0,T,w_{\kappa})& \hookrightarrow C_{\nicefrac{\kappa}{p}}((0,T];X_{1-\frac{1}{p},p}),
\end{align}
where the embedding constants can be taken independent of $T$ if one restricts to functions that vanish at $t=0$. Here,  $C_{\nicefrac{\kappa}{p}}((0,T];Y)$ is the space of continuous functions   $u\col (0,T]\to Y$ for which $\sup_{t\in (0,T]}t^{\nicefrac{\kappa}{p}}\|u(t)\|_{Y}<\infty$.

Although $A$ is allowed to be unbounded in time, we will often suppose that there is a constant $C_A$ such that 
\begin{align}\label{eq:bddcondLpA}
\|Av\|_{L^p(0,T,w_{\kappa};X_0)}\leq C_A \|v\|_{\MR^p(0,T,w_{\kappa})}, \qquad v\in \MR^p(0,T,w_{\kappa}). 
\end{align}
Due to this condition, it is clear that for $v\in \MR^p(0,T,w_{\kappa})$ one has $v',Av\in L^p(0,T,w_{\kappa};X_0)$ and 
\begin{align*}
\|v'+Av\|_{L^p(0,T,w_{\kappa};X_0)} &\leq \|v'\|_{L^p(0,T,w_{\kappa};X_0)} + \|Av\|_{L^p(0,T,w_{\kappa};X_0)}
\leq (1+C_A) \|v\|_{\MR^p(0,T,w_{\kappa})}.
\end{align*}
Maximal regularity theory addresses the validity of the converse estimate: 
  \begin{definition}\label{def:MR}
Let $T\in (0,\infty)$. Let $X_0$ and $X_1$ be Banach spaces such that $X_1\hookrightarrow X_0$ continuously. 
Suppose that $A\col[0,T]\to \calL(X_1, X_0)$ is strongly measurable in the strong operator topology.  
Let $p\in (1, \infty)$, $\kappa\in [0,p-1)$. 

The operator family $A$ is said to have \emph{maximal $L^p(0,T,w_{\kappa};X_0)$-regularity} if 
for every $f\in L^p(0,T,w_{\kappa};X_0)$ there exists a unique strong solution $u\in \MR^p(0,T,w_{\kappa})$ to \eqref{eq:parabolic_problem_prel} with $u_0=0$, and there is a constant $C$ independent of $f$ such that for any $\tau\in (0,T]$ and any strong solution $v\in \MR^p(0,\tau,w_{\kappa})$ to \eqref{eq:parabolic_problem_prel} with $u_0=0$  one has
\begin{equation}\label{eq:MRest}
\|v\|_{\MR^p(0,\tau,w_{\kappa})} \leq C\|f\|_{L^p(0,\tau,w_{\kappa};X_0)}.
\end{equation}
The infimum of all admissible constants $C$ will be denoted by $M_{p,\kappa,A}(0,T)$ and is called the \emph{maximal $L^p(0,T,w_{\kappa};X_0)$-regularity constant}. 

If the underlying space $X_0$ is clear from the context, we will write `maximal $L^p(0,T,w_{\kappa})$-regularity'.
\end{definition}

The definition above extends to subintervals $(a,b)\subseteq (0,T)$ by  replacing $(0,T)$ by $(a,b)$, $u(0)$ by $u(a)$ and taking $\tau\in(a,b]$.

Note that the existence and uniqueness for all $f\in L^p(0,T,w_{\kappa};X_0)$  require $X_1$ to be chosen appropriately.  
Also, one can conclude   $A u\in L^p(0,T,\wk;X_0)$ not only from \eqref{eq:bddcondLpA}, but  also from the identity $A u = f-u'$ that holds a.e.\ \cite[Lem.\ 2.5.8]{HNVWvolume1}.

Since we incorporated some flexibility into the above definition by allowing for arbitrary $\tau$ in \eqref{eq:MRest}, we obtain the following result on restriction to subintervals. Without the flexibility with $\tau$ in \eqref{eq:MRest}, it is unclear whether the uniqueness of solutions extends to subintervals $(0,\tau)$. In the autonomous setting, these problems can be solved in a different way (see \cite[\S17.2]{HNVWvolume3}).  Some discussion on the nonautonomous setting can also be found in \cite{ACFP}.

We consider the following problem on $(a,b)\subset (0,T)$: 
\begin{equation}
\label{eq:parabolic_problem_prel subinterval}
\left\{
\begin{aligned}
    u' + A u &= f, \\
    u(a) &= u_a.
\end{aligned}
\right.
\end{equation}

\begin{lemma}[Change of interval]\label{lem:restriction}
Let Assumption \ref{assump:XA} be satisfied. Let $p\in (1, \infty)$ and $\kappa\in [0,p-1)$. Suppose that 
$A$ has maximal $L^p(0,T,w_{\kappa})$-regularity.  
For any $0\leq a< b \leq T$ and $f\in L^p(a,b,w_{\kappa};X_0)$, there exists a unique strong solution $u\in \MR^p(a,b,w_{\kappa})$ to \eqref{eq:parabolic_problem_prel subinterval} with $u_a=0$, and   
\begin{equation*}
\|u\|_{\MR^p(a,b,w_{\kappa})}\leq M_{p,\kappa,A}(0,T) \|f\|_{L^p(a,b,w_{\kappa};X_0)}.
\end{equation*}
In particular,  $A$ has maximal $L^p(a,b,w_{\kappa})$-regularity and 
$M_{p,\kappa,A}(a,b)\leq M_{p,\kappa,A}(0,T)$.
 \end{lemma}
\begin{proof}
We first prove uniqueness. Let $u\in \MR^p(a,b,w_{\kappa})$ be a strong solution to 
\eqref{eq:parabolic_problem_prel subinterval} with $u_a=0$.  
Extending $u$ by zero on $(0,a)$, one has $u\in \MR^p(0,b,w_{\kappa})$ and $u$ is a strong solution to \eqref{eq:parabolic_problem_prel} on $(0,b)$ with $u_0=0$. Thus \eqref{eq:MRest} implies that $u=0$. 

For existence let $f\in L^p(a,b,w_{\kappa};X_0)$. Extend $f$ by zero outside $(a,b)$. Let $u\in \MR^p(0,T,w_{\kappa})$ be the corresponding strong solution to \eqref{eq:parabolic_problem_prel} with $u_0=0$. Then $u\in \MR^p(a,b,w_{\kappa})$, $u$ solves   \eqref{eq:parabolic_problem_prel subinterval} with $u_a=0$, and for every constant $C$  admissible for \eqref{eq:MRest},  we have 
\begin{equation*}
\|u\|_{\MR^p(a,b,w_{\kappa})}\leq \|u\|_{\MR^p(0,b,w_{\kappa})} \leq C\|f\|_{L^p(0,b,w_{\kappa};X_0)} = C\|f\|_{L^p(a,b,w_{\kappa};X_0)}.
\end{equation*}
\end{proof}

The following lemma is standard (see \cite[Prop.\ 3.10]{AV22nonlinear1} and \cite[Lem.\ 2.2]{ACFP}) and holds in our generalized setting with only minor modifications.
As a consequence of the lemma, we can always assume $u_0=0$ when analyzing \eqref{eq:parabolic_problem_prel}.  
 
\begin{lemma}\label{lem:initial}
Let Assumption \ref{assump:XA}  be satisfied.
Let $p\in (1, \infty)$ and $\kappa\in [0,p-1)$.  Suppose that \eqref{eq:bddcondLpA} holds  and $A$ has maximal $L^p(0,T,w_{\kappa})$-regularity with constant $M\ceqq M_{p,\kappa,A}(0,T)$. Then the following assertions hold:
\begin{enumerate}[\em (1)]
\item Let $0<b\leq T$. For every $u_0\in X_{1-\frac{1+\kappa}{p},p}$  and $f\in L^p(0,b,w_{\kappa};X_0)$ there exists a unique strong solution $u\in \MR^p(0,b,w_{\kappa})$ to \eqref{eq:parabolic_problem_prel}, and  
\begin{align*}
&\|u\|_{\MR^p(0,b,w_{\kappa})}\leq  \tfrac{12p}{1+\kappa}\big(M(1+ C_A)+1\big) \|u_0\|_{X_{1-\frac{1+\kappa}{p},p}}+M\|f\|_{L^p(0,b,w_{\kappa};X_0)}. 
\end{align*}
\item Let $0<a<b\leq T$. For every $u_a\in X_{1-\frac{1}{p},p}$ and $f\in L^p(a,b,w_{\kappa};X_0)$ there exists a unique strong solution $u\in \MR^p(a,b,w_{\kappa})$ to \eqref{eq:parabolic_problem_prel subinterval}, and 
\begin{equation*}
\|u\|_{\MR^p(a,b,w_{\kappa})}\leq 12pT^{\kappa/p}\big(M(1 + C_A)+1\big) \|u_a\|_{X_{1-\frac{1}{p},p}}+ M \|f\|_{L^p(a,b,w_{\kappa};X_0)}.
\end{equation*}
\end{enumerate}
\end{lemma} 
\begin{proof}
(1): \ Uniqueness is clear. To prove the existence and estimate, we exploit the trace method for real interpolation (see \cite[{\S}L.2]{HNVWvolume3} and \cite[\S 1.8]{Tri78}). Let $z\in \MR^p(0,\infty,w_{\kappa})$ be such  that $z(0) = u_0$. 
Note that $z',Az\in L^p(0,b,\wk;X_0)$, using \eqref{eq:bddcondLpA}. 
Let $v\in \MR^p(0,b,w_{\kappa})$ be the unique solution to the problem $v' + A v= f- z'-Az$ with $v(0) = 0$. Then $u \ceqq v+z$ satisfies \eqref{eq:parabolic_problem_prel}.  
Moreover, 
\begin{align*}
\|u&\|_{\MR^p(0,b,w_{\kappa})} \leq \|v\|_{\MR^p(0,b,w_{\kappa})}+ \|z\|_{\MR^p(0,T,w_{\kappa})} 
\\ & \leq M \big(\|f\|_{L^p(0,b,w_{\kappa};X_0)}+   \|z'\|_{L^p(0,T,w_{\kappa};X_0)}+ \|A z\|_{L^p(0,T,w_{\kappa};X_0)}\big) + \|z\|_{\MR^p(0,T,w_{\kappa})}  
\\ & \leq M \|f\|_{L^p(0,b,w_{\kappa};X_0)} + \big(M(1+ C_A)+1\big) \|z\|_{\MR^p(0,\infty,w_{\kappa})},
\end{align*} 
applying \eqref{eq:bddcondLpA} in the last line.
Taking the infimum over all $z\in \MR^p(0,\infty,w_{\kappa})$, the stated estimate  follows from the trace method for real interpolation, see \cite[Th.\ L.2.3]{HNVWvolume3} for the constant. 

(2): \ Uniqueness is clear  from Lemma \ref{lem:restriction}. To prove existence, we use the same strategy as above. Let $z\in \MR^p(a,\infty)$ be such that  $z(a) = u_a$. We repeat the construction of $v$ of (1), but this time on $(a,b)$, and we again put $u\ceqq v+z$. By Lemma \ref{lem:restriction} and a computation similar to that in part (1), we have  
\begin{align*}
\|u\|_{\MR^p(a,b,w_{\kappa})}  &\leq M \|f\|_{L^p(a,b,w_{\kappa};X_0)} + \big(M(1+ C_A)+1\big) \|z\|_{\MR^p(a,T,\wk)}
\\ & \leq M \|f\|_{L^p(a,b,w_{\kappa};X_0)}+ T^{\kappa/p}\big(M(1+ C_A)+1\big) \|z\|_{\MR^p(a,\infty)}.
\end{align*} 
Taking the infimum over all $z\in \MR^p(a,\infty)$, the claim follows as before and by a simple shift. 
\end{proof}

A consequence of the theory above is that if $A$ satisfies \eqref{eq:bddcondLpA} and has maximal $L^p(0,T,w_{\kappa})$-regularity then the following mapping is an isomorphism: 
\[
\mathcal{J}\col X_{1-\frac{1+\kappa}{p},p}\times L^p(0,T,w_{\kappa};X_0)\to \MR^p(0,T,w_{\kappa}),  \ \ \mathcal{J} (u_0, f) = u, 
\] 
where $u$ is the strong solution to \eqref{eq:parabolic_problem_prel}. 
Indeed, \eqref{eq:traceembedding0T} and \eqref{eq:bddcondLpA} give respectively 
\begin{align*}
&\|u_0\|_{X_{1-\frac{1+\kappa}{p},p}}\leq C \|u\|_{\MR^p(0,T,w_{\kappa})}, \\
&\|f\|_{L^p(0,T,w_{\kappa};X_0)} \leq (1+C_A)\|u\|_{\MR^p(0,T,w_{\kappa})}.
\end{align*} 
Combining this with Lemma \ref{lem:initial}(1), we obtain the required two-sided estimates showing that $\mathcal{J}$ is an isomorphism. 
 
Many standard results of the autonomous setting fail for the setting of time-dependent $A$. Some of them we collect in the next remark. 
\begin{remark}\ 
\begin{enumerate}[(1)]
\item Extrapolation in $p$ fails in general. 
Counterexamples to extrapolation for singular families $A$ appear in \cite{prussschnaubelt01}. Counterexamples also exist in the uniformly bounded setting, for instance with differential operators in divergence form, as considered in \cite{BMV, Krylov1d}. Therein, the coefficients are bounded and uniformly elliptic, and measurable in $(t,x)$. 
\item Extrapolation in $p$ can be obtained if $A$ is (relatively) continuous in time (see \cite{ACFP,prussschnaubelt01}), in which case one can reduce to the autonomous setting. This is, for instance, considered in \cite[Th.\ 17.2.26]{HNVWvolume3}.  
\item Maximal $L^p$-regularity does not seem to imply that $A$ generates an evolution family on $X_0$. Indeed, in the case $A$ depends on time in a measurable way, maximal $L^p$-regularity is known to hold in various situations (cf.\ \cite{DongKim11,GV17_pot}), but it is unclear whether $A$ generates an evolution family on $X_0$. An evolution system can be constructed on certain real interpolation spaces via Lemma \ref{lem:initial}. 
        Moreover, for $\kappa=0$, and continuous $t\mapsto A(t)\in \calL(X_1, X_0)$, an evolution family does exist as was shown in \cite[Th.\ 2.5]{prussschnaubelt01}, but even that case is highly nontrivial.  
\end{enumerate}
\end{remark}

\subsection{Main result}
We now present the main result of this section. It contains Theorem \ref{thm:MRI} as a special case.   
\begin{theorem}\label{thm:perturbLp}
Let $p\in (1, \infty)$, $\kappa\in [0,p-1)$ and $q\in (p, \infty)$. Suppose that Assumption \ref{assump:XA} and \eqref{eq:bddcondLpA} hold and that $A$ has maximal $L^p(0,T,w_{\kappa})$-regularity. Let $B\col [0,T]\to \calL(X_{1-\nicefrac{1}{q},1},X_0)$ be strongly measurable in the strong operator topology and suppose that 
$\|B(\cdot)\|_{\calL(X_{1-\nicefrac{1}{q},1},X_0)}\leq b(\cdot)$, where $b\in L^q(0,T)$. Then $A+B$ has maximal $L^p(0,T,w_\kappa)$-regularity with $M_{p,\kappa,A+B}(0,T)\leq C$, where $C$ only depends on $p, q,\kappa,X_0, X_1, M_{p,\kappa,A}(0,T)$ and $\|b\|_{L^q(0,T)}$.
\end{theorem}

Recall that nonzero initial values can be added as a consequence of Lemma \ref{lem:initial}. 
From Step 0 of the proof below, it follows that $B$ satisfies 
\[\|B v\|_{L^p(0,t,w_{\kappa};X_0)} \leq C \|v\|_{\MR^{p}(0,t,w_{\kappa})}, \qquad v\in \MR^{p}(0,t,w_{\kappa}), t\in(0,T],
\]
see  \eqref{eq:KeyestB}. Consequently, $A+B$ also has the  mapping property \eqref{eq:bddcondLpA}. 

For relatively bounded perturbations $B$, maximal $L^p$-regularity of $A+B$ is typically established via the method of continuity \cite[Prop.\ 17.2.49]{HNVWvolume3}. This approach appears inapplicable in our setting. 
Consequently, the proof of Theorem \ref{thm:perturbLp} utilizes a gluing procedure with nonstandard features. In particular, due to the weight, it is nontrivial to show that the constant $C$ only depends on $\|b\|_{L^q(0,T)}$ and not on the shape of the function $b$ itself. The dependence on $\|b\|_{L^q(0,T)}$ is exponential, and this is sharp in general. In Theorem \ref{thm:perturbGronwall} we will treat the case $p=q$, which turns out to be much simpler.  

A repeated application of Theorem \ref{thm:perturbLp} further extends the result to $A+\sum_{i=1}^n B_i$, where $\|B_i(t)\|_{\calL(X_{1-\nicefrac{1}{q_i},1},X_0)}\leq b_i(t)$ and $b_i\in L^{q_i}(0,T)$ with $q_i\in (p, \infty)$.

\begin{proof}[Proof of Theorem \ref{thm:perturbLp}]
{\em Step 0:} As $T, p, \kappa,$ and $A$ are fixed throughout the proof, we simplify notation by writing $M$ for $M_{p,\kappa,A}(0,T)$. 
Let $\prescript{}{0}{\MR}^{p}(s,t,w_{\kappa})$ be the subspace  of $u\in \MR^{p}(s,t,w_{\kappa})$ such that $u(s) = 0$. 

Let $t\in (0,T]$. For $v\in \prescript{}{0}{\MR}^{p}(0,t,w_{\kappa})$, we have by H\"older's inequality: 
\[\|B v\|_{L^p(0,t,w_{\kappa};X_0)}\leq \big\| b \|v\|_{X_{1-\nicefrac{1}{q},1}} \big\|_{L^p(0,t,w_{\kappa})}\leq \|b\|_{L^q(0,t)} 
\|v w_{\nicefrac{\kappa}{p}}\|_{L^{\frac{pq}{q-p}}(0,t;X_{1-\nicefrac{1}{q},1})}. \]  
Note that $(X_{1-\frac{1}{p},p},X_1)_{1-\frac{p}{q},1} =X_{1-\frac{1}{q},1}$  by reiteration (see \cite[Th.\ L.3.1]{HNVWvolume3}), so 
\begin{align*}
\|v w_{\nicefrac{\kappa}{p}}\|_{X_{1-\frac{1}{q},1}}\lesssim \|v w_{\nicefrac{\kappa}{p}}\|_{X_{1-\frac1p,p}}^{\frac{p}{q}}  \|v w_{\nicefrac{\kappa}{p}}\|_{X_{1}}^{1-\frac{p}{q}}.
\end{align*}
Combined with \eqref{eq:traceembeddingaT},   
we obtain the following estimates with $t$-independent constants (since $v\in \prescript{}{0}{\MR}^{p}(0,t,w_{\kappa})$):  
\begin{align*}
\|v w_{\nicefrac{\kappa}{p}}\|_{L^{\frac{pq}{q-p}}(0,t;X_{1-\frac{1}{q},1})}& \lesssim  \sup_{s\in [0,t]}\|v(s) w_{\nicefrac{\kappa}{p}}(s)\|_{X_{1-\frac1p,p}}^{\frac{p}{q}}  \|v w_{\nicefrac{\kappa}{p}}\|_{L^p(0,t;X_{1})}^{1-\frac{p}{q}}
\\ & \lesssim \|v\|_{{\MR}^{p}(0,t,w_{\kappa})}^{\frac{p}{q}}  \|v\|_{{\MR}^{p}(0,t,w_{\kappa})}^{1-\frac{p}{q}} = \|v\|_{{\MR}^{p}(0,t,w_{\kappa})}.
\end{align*}
We conclude  that for a constant $C_0$ depending on $X_0, X_1,p,q$ and $\kappa$, 
\begin{align}\label{eq:KeyestB}
\|B v\|_{L^p(0,t,w_{\kappa};X_0)} \leq C_0 \|b\|_{L^q(0,t)}   \|v\|_{{\MR}^{p}(0,t,w_{\kappa})}.
\end{align}
The estimates above also hold for $v\in {\MR}^{p}(0,t,w_{\kappa})$, but in the latter case, the constant would depend on $t$. 

\medskip

The strategy for the remainder of the proof is to construct local solutions on small time intervals and glue them together. In Step 1, the initial condition is zero, but in Step 2, we address nonzero initial conditions as well. 

{\em Step 1:}
Let $\tau$ be such that $\|b\|_{L^q(0,\tau)} = \frac{1}{2 C_{0} M}$ if possible, and otherwise set $\tau = T$.  In the latter case, no further steps are needed. 

Let $\Phi\col\prescript{}{0}{\MR}^{p}(0,\tau,w_{\kappa})\to \prescript{}{0}{\MR}^{p}(0,\tau,w_{\kappa})$ be given by $\Phi(v) = u$, where $u$ solves 
\[u' + A u = - Bv+f.\]
By the maximal $L^p(0,T,w_\kappa)$-regularity of $A$ and Lemma \ref{lem:restriction}, the mapping $\Phi$ is well-defined and with \eqref{eq:KeyestB}, we have 
\begin{align*}
\|\Phi(v_1) - \Phi(v_2)\|_{{\MR}^{p}(0,\tau,w_{\kappa})}&\leq M\|B(v_1 - v_2)\|_{L^p(0,\tau,w_{\kappa};X_0)}
\\ & \leq M C_{0} \|b\|_{L^q(0,\tau)}   \|v_1-v_2\|_{{\MR}^{p}(0,\tau,w_{\kappa})}
\\ & \leq \frac12\|v_1-v_2\|_{{\MR}^{p}(0,\tau,w_{\kappa})}.
\end{align*}
By the contraction mapping theorem there exists a unique $u\in \prescript{}{0}{\MR}^{p}(0,\tau,w_{\kappa})$ such that $u= \Phi(u)$, which implies that there is a unique strong solution $u\in \prescript{}{0}{\MR}^{p}(0,\tau,w_{\kappa})$ to $u' + (A+B) u = f$. Moreover, 
\begin{align*}
\|u\|_{{\MR}^{p}(0,\tau,w_{\kappa})}   = \|\Phi(u)\|_{{\MR}^{p}(0,\tau,w_{\kappa})}
&\leq \|\Phi(u)- \Phi(0)\|_{{\MR}^{p}(0,\tau,w_{\kappa})} + \|\Phi(0)\|_{{\MR}^{p}(0,\tau,w_{\kappa})}
\\ & \leq \frac12 \|u\|_{{\MR}^{p}(0,\tau,w_{\kappa})} + M\|f\|_{L^p(0,\tau,w_{\kappa};X_0)}.
\end{align*}
Thus 
\begin{equation}\label{eq: n=1 estimate}
\|u\|_{{\MR}^{p}(0,\tau,w_{\kappa})} \leq 2M\|f\|_{L^p(0,\tau,w_{\kappa};X_0)}. 
\end{equation}

\medskip
{\em Step 2: } We proceed iteratively. 

Let $\tau_1 \ceqq \tau$. Let $n\geq 1$ be given and suppose that $0<\tau_{1}< \ldots< \tau_n<T$ are given and $u\in \prescript{}{0}{\MR}^{p}(0,\tau_n,w_{\kappa})$ is the unique strong solution within $\prescript{}{0}{\MR}^{p}(0,\tau_n,w_{\kappa})$ to $u'+A u +Bu = f$. We will extend the solution uniquely to $u\in \prescript{}{0}{\MR}^{p}(0,\tau_{n+1},w_{\kappa})$ defined on  $[0,\tau_{n+1}]$.  Additionally,  we will establish a recursive energy estimate that will be utilized in Step 3. 

Choose $\tau_{n+1}$ such that $\|b\|_{L^q(\tau_{n},\tau_{n+1})} = \frac{1}{2 C_{0} M}$ if possible, and otherwise set $\tau_{n+1} \ceqq T$ and turn to Step 3. 
From \eqref{eq:traceembeddingaT} we know that $u(\tau_n)\in X_{1-\frac1p,p}$. By Lemma \ref{lem:initial} there exists a unique strong solution $\wt{\zz}\in \MR^p(\tau_n, \tau_{n+1},w_{\kappa})$ to 
\begin{equation*} 
\left\{
\begin{aligned}
\wt{\zz}\kern 0.09em' + A\wt{\zz} &= 0,
\\ \wt{\zz}(\tau_n) & = u(\tau_n). 
\end{aligned}
\right.
\end{equation*}
Let $\Phi_n\col\prescript{}{0}{\MR}^p(0, \tau_{n+1},w_{\kappa})\to \prescript{}{0}{\MR}^p(0, \tau_{n+1},w_{\kappa})$ be given by $\Phi_n(v) = \zz$, where $\zz$ solves
\begin{equation*} 
\left\{
\begin{aligned}
\zz' + A\zz &= -B(v)\one_{[\tau_n, \tau_{n+1}]} - B(\wt{\zz})\one_{[\tau_n, \tau_{n+1}]} + f \one_{[\tau_n, \tau_{n+1}]},
\\ \zz(0) & = 0. 
\end{aligned}
\right.
\end{equation*}
By maximal $L^p(0,T,w_\kappa)$-regularity, $\Phi_n$ is well-defined, and by \eqref{eq:KeyestB}, 
\begin{align*}
\|\Phi_n(v_1) - \Phi_n(v_2)\|_{\MR^p(0,\tau_{n+1},w_{\kappa})} & \leq M \|B(v_1 - v_2)\|_{L^p(\tau_n,\tau_{n+1},w_{\kappa};X_0)}
\\ & \leq M C_{0} \|b\|_{L^q(\tau_n, \tau_{n+1})}   \|v_1-v_2\|_{\MR^p(0,\tau_{n+1},w_{\kappa})}
\\ & \leq \frac12\|v_1-v_2\|_{\MR^p(0,\tau_{n+1},w_{\kappa})}.
\end{align*}
Therefore, there exists a unique $\zz\in \prescript{}{0}{\MR}^p(0, \tau_{n+1},w_{\kappa})$ such that $\zz = \Phi_n(\zz)$. 
Moreover, 
\begin{align*}
\|\zz\|_{\MR^p(0, \tau_{n+1},w_{\kappa})}
 & \leq \|\Phi_n(\zz) - \Phi_n(0)\|_{\MR^p(0, \tau_{n+1},w_{\kappa})}+\|\Phi_n(0)\|_{\MR^p(0, \tau_{n+1},w_{\kappa})}
\\ & \leq \frac12\|\zz\|_{\MR^p(0, \tau_{n+1},w_{\kappa})}+M\|B \wt{\zz}\|_{L^p(\tau_n,\tau_{n+1},w_{\kappa};X_0)} +M\|f\|_{L^p(0,\tau_{n+1},w_{\kappa};X_0)}.
\end{align*}
Thus 
\[\|\zz\|_{\MR^p(0, \tau_{n+1},w_{\kappa})}\leq 2M\|B \wt{\zz}\|_{L^p(\tau_n,\tau_{n+1},w_{\kappa};X_0)}+2M\|f\|_{L^p(0,\tau_{n+1},w_{\kappa};X_0)}.\]
Now extend $u$ by defining $$u = \zz|_{[\tau_n, \tau_{n+1}]}+\wt{\zz}\quad\text{ on $[\tau_n,\tau_{n+1}]$}.$$
Then $u\in \MR^p(\tau_n, \tau_{n+1},w_{\kappa})$  is a strong solution to 
\begin{equation*} 
\left\{
\begin{aligned}
y' + Ay + By &=   f,  
\\ y(\tau_n) & = u(\tau_n). 
\end{aligned}
\right.
\end{equation*}
Moreover, the   estimates above yield 
\begin{align}\label{eq: est u}
\|u&\|_{\MR^p(\tau_n, \tau_{n+1},w_{\kappa})} \leq \|\wt{\zz}\|_{\MR^p(\tau_n, \tau_{n+1},w_{\kappa})} + \|\zz\|_{\MR^p(\tau_n, \tau_{n+1},w_{\kappa})}
\\ 
& \leq \|\wt{\zz}\|_{\MR^p(\tau_n, \tau_{n+1},w_{\kappa})} + 2M\|B \wt{\zz}\|_{L^p(\tau_n,\tau_{n+1},w_{\kappa};X_0)} + 2M\|f\|_{L^p(0,\tau_{n+1},w_{\kappa};X_0)}. \notag
\end{align}
It remains to bound the norms of $\wt{\zz}$ and $B\wt{\zz}$. To this end, consider $v\in \MR^p(0,\tau_{n+1},w_{\kappa})$ defined by $v=u$ on $[0,\tau_n]$  and $v = \wt{\zz}$ on $[\tau_n, \tau_{n+1}]$. Then $v$ is a strong solution to 
\begin{equation*}  
\left\{
\begin{aligned}
v'+A v   &= - B v \one_{[0,\tau_n]}+f \one_{[0,\tau_n]},
\\ v(0) & = 0,  
\end{aligned}
\right.
\end{equation*}
and by maximal $L^p(0,T,w_\kappa)$-regularity and \eqref{eq:KeyestB}, we have 
\begin{align*} \|v\|_{\MR^p(0, \tau_{n+1},w_{\kappa})}
 & \leq M\|Bv\|_{L^p(0,\tau_n,w_{\kappa};X_0)} + M \|f\|_{L^p(0,\tau_n,w_{\kappa};X_0)}
\\ & \leq M C_{0} \|b\|_{L^q(0,\tau_{n})}   \|v\|_{\MR^p(0,\tau_n,w_{\kappa})} + M \|f\|_{L^p(0,\tau_n,w_{\kappa};X_0)}\\
&\leq \frac{n}{2}\|u\|_{\MR^p(0,\tau_n,w_{\kappa})} + M \|f\|_{L^p(0,\tau_n,w_{\kappa};X_0)}.
\end{align*}
For the last inequality above, we use that $v=u$ on $[0,\tau_n]$, and by construction, 
\begin{equation*} 
  \|b\|_{L^q(0,\tau_{n})} =  \frac{n^{1/q}}{2 C_{0} M}\leq \frac{n}{2 C_{0} M}.
\end{equation*}

Now we estimate  $ \wt{\zz}$ and $B\wt{\zz}$ in terms of $v$. 
For the $\wt{\zz}$-term, we have 
\begin{align*}
\|\wt{\zz}\|_{\MR^p(\tau_n, \tau_{n+1},w_{\kappa})} & \leq \|v\|_{\MR^p(0, \tau_{n+1},w_{\kappa})},
\end{align*}
and for the $B\wt{\zz}$-term, by \eqref{eq:KeyestB} applied to $\one_{[\tau_n, \tau_{n+1}]}B$, we have \begin{align*}
 \|B\wt{\zz}\|_{L^p(\tau_n,\tau_{n+1},w_{\kappa};X_0)} & \leq \|\one_{[\tau_n, \tau_{n+1}]}Bv\|_{L^p(0,\tau_{n+1},w_{\kappa};X_0)}\\ 
  &\leq 
 C_{0} \|b\|_{L^q(\tau_{n},\tau_{n+1})} \|v\|_{\MR^p(0,\tau_{n+1},w_{\kappa})}\\
 &\leq (2M)^{-1}\|v\|_{\MR^p(0,\tau_{n+1},w_{\kappa})}.
  \end{align*}

Combining \eqref{eq: est u} with the estimates for $B\wt{\zz}$, $\wt{\zz}$  and $v$, we conclude that 
\begin{align}
\|u\|_{\MR^p(0, \tau_{n+1},w_{\kappa})}&\leq 
\|u\|_{\MR^p(0, \tau_n,w_{\kappa})}+
\|u\|_{\MR^p(\tau_n, \tau_{n+1},w_{\kappa})}\notag\\
&\leq \|u\|_{\MR^p(0, \tau_n,w_{\kappa})}+2\|v\|_{\MR^p(0,\tau_{n+1},w_{\kappa})} 
+ 2M\|f\|_{L^p(0,\tau_{n+1},w_{\kappa};X_0)}\notag\\
&\leq (1+n)\|u\|_{\MR^p(0,\tau_n,w_{\kappa})} + 4M\|f\|_{L^p(0,\tau_{n+1},w_{\kappa};X_0)}. \label{eq: recursive estimate}
\end{align}

Note that the extended $u$ is a strong solution on $[0, \tau_{n+1}]$. Moreover, $u$ is continuous at $\tau_n$, so by \cite[Prop.\ 2.5.9]{HNVWvolume1}, we conclude that $u\in \MR^p(0, \tau_{n+1},w_{\kappa})$. 
To complete the induction step, it remains to prove that $u$ is the  unique strong solution  within $\MR^p(0,\tau_{n+1},\wk)$. Let $u_1,u_2\in \MR^p(0,\tau_{n+1},\wk)$ both be strong solutions. By uniqueness of the solution $u$ on $[0,\tau_n]$, we find $u_1=u_2$ on $[0,\tau_n]$ and  $u_1(\tau_n)=u_2(\tau_n)=u(\tau_n)$ in $\Xtr$ by \eqref{eq:traceembedding0T}. Therefore, defining $w_i \ceqq (u_i -\widetilde{z})\one_{[\tau_n,\tau_{n+1}]}$ for $i=1,2$, we have $w_i  \in \MR^p(0,\tau_{n+1},\wk)$ (since $w_i(\tau_n)=0$) and $w_i=\Phi_n(w_i)$. By uniqueness of the fixed point of $\Phi_n$, we can conclude that $w_1=w_2$, hence $u_1=u_2$ on $[0,\tau_{n+1}]$. 

{\em Step 3:} Let $N$ be the least integer such that $\tau_N = T$. Then from Steps 1 and 2, we can conclude that there exists a unique strong solution $u\in \prescript{}{0}{\MR}^p(0,T,w_{\kappa})$ to $u' +A u +B u  = f$. 
Moreover, the recursive estimate \eqref{eq: recursive estimate} together with \eqref{eq: n=1 estimate} yield
\[
\|u\|_{\MR^p(0,T,w_{\kappa})} \leq 4M\textstyle{\sum_{j=1}^N} \frac{N!}{j!}\|f\|_{L^p(0,T,w_{\kappa};X_0)} \leq 7 N! M\|f\|_{L^p(0,T,w_{\kappa};X_0)}, 
\]
using that $ \sum_{j=1}^N \frac{1}{j!} \leq  e-1 \leq \frac74$. 

It remains to bound $N$ in terms of $\|b\|_{L^q(0,T)}$. To do so, set  $\tau_0 \ceqq 0$ and note that
\begin{align*}
\|b\|_{L^q(0,T)}^q \geq \sum_{n=1}^{N-1} \|b\|_{L^q(\tau_{n-1},\tau_n)}^q = \frac{N-1}{(2 C_{0} M)^q}. 
\end{align*}
Therefore, $N\leq 1+ (2 C_{0} M)^q \|b\|_{L^q(0,T)}^q\eqqc K$. Hence, using the $\Gamma$-function, the required bound holds with  constant  
\[C = 7\Gamma(K+1) M.\] 
\end{proof}

\subsection{Limiting case $p=q$}

The limiting case $q = p$ is significantly more straightforward and is known to experts, at least in the unweighted setting.  We will assume that each $B(t)$ is an element of $\calL(X_{1-\nicefrac{1}{p},p},X_0)$, which is more restrictive than the condition in Theorem \ref{thm:perturbLp}, since $q>p$. On the other hand, we allow for $b\in L^p(0,T)$. 
\begin{theorem}\label{thm:perturbGronwall}
Let $p\in (1, \infty)$ and $\kappa\in [0,p-1)$. Suppose that Assumption \ref{assump:XA} and \eqref{eq:bddcondLpA} hold and that $A$ has maximal $L^p(0,T,w_{\kappa})$-regularity. Let $B\col[0,T]\to \calL(X_{1-\nicefrac{1}{p},p},X_0)$ be strongly measurable in the strong operator topology and suppose that 
$\|B(t)\|_{\calL(X_{1-\nicefrac{1}{p},p},X_0)}\leq b(t)$, where $b\in L^p(0,T)$. Then $A+B$ has maximal $L^p(0,T,w_\kappa)$-regularity with $M_{p,\kappa,A+B}(0,T)\leq C$, where $C$ only depends on $p, \kappa, X_0, X_1,  M_{p,\kappa,A}(0,T)$ and $\|b\|_{L^p(0,T)}$.
\end{theorem}
We include a short proof based on the method of continuity and Gronwall's lemma. 
\begin{proof}
Fix an arbitrary $\lambda \in [0,1]$. Suppose that $u\in \MR^p(0,\tau,w_{\kappa})$ is a solution to 
\begin{equation*}
\left\{
\begin{aligned}
    u' + A u + \lambda B u &= f, \\
    u(0) &= 0.
\end{aligned}
\right.
\end{equation*}
Observe that 
\begin{align}\label{eq:Bu in Lp}
  \|B u\|_{L^p(0,t,w_{\kappa};X_0)}^p \leq  \int_0^t b(s)^p s^{\kappa} \|u(s)\|_{X_{1-\nicefrac{1}{p},p}}^p \dd s<\infty,
\end{align}
where the finiteness holds by \eqref{eq:traceembeddingaT} and since $b\in L^p(0,T)$.  

By maximal $L^p(0,T,w_{\kappa})$-regularity of $A$ with constant $M$, and by applying \eqref{eq:traceembeddingaT} (with embedding constant uniform in $\tau$ for functions vanishing at zero), we thus have for any $t\in (0,T]$: 
\begin{align*}
\|u\|_{\MR^p(0,t,w_{\kappa})} + \|u\|_{C_{\nicefrac{\kappa}{p}}((0,t];X_{1-\frac{1}{p},p})} \leq C M \|B u\|_{L^p(0,t,w_{\kappa};X_0)} + C M\|f\|_{L^p(0,t,w_{\kappa};X_0)}.
\end{align*}
Taking $p$-th powers and using \eqref{eq:Bu in Lp}, it follows that 
\begin{align*}
& \|u\|_{\MR^p(0,t,w_{\kappa})}^p + \|u\|_{C_{\nicefrac{\kappa}{p}}((0,t];X_{1-\frac{1}{p},p})}^p\\  
& \leq 2^{p-1} C^p M^p \|f\|_{L^p(0,t,w_{\kappa};X_0)}^p + 2^{p-1} C^p M^p\int_0^t b(s)^p s^{\kappa} \|u(s)\|_{X_{1-\nicefrac{1}{p},p}}^p \dd s. 
\end{align*}
By Gronwall's lemma applied to $\phi(s) \ceqq \|u\|_{\MR^p(0,s,w_{\kappa})}^p + s^{\kappa} \|u(s)\|^p_{X_{1-\nicefrac{1}{p},p}}$, it follows that for all $t\in [0,T]$: 
\begin{align*}
\phi(t) \leq 2^{p-1} C^pM^p \|f\|_{L^p(0,t,w_{\kappa};X_0)}^p \exp(2^{p-1} C^p M^p \|b\|_{L^p(0,t)}^p). 
\end{align*}
In particular, this shows that 
\begin{align}\label{eq:est Llambda}
\|u\|_{\MR^p(0,t,w_{\kappa})} \leq 2 C M \exp(p^{-1} 2^{p-1} C^pM^p \|b\|_{L^p(0,t)}^p) \|f\|_{L^p(0,t,w_{\kappa};X_0)}, 
\end{align}
which is independent of $\lambda$. 

It remains to apply the method of continuity (see \cite[Lem.\ 16.2.2]{HNVWvolume3}) to the mapping $L_{\lambda}:\prescript{}{0}{\MR}^{p}(0,\tau,w_{\kappa})\to L^p(0,\tau,w_{\kappa};X_0)$ given by $L_{\lambda} u = u'+Au +\lambda B u$, where $\tau\in (0,T]$ is fixed. This gives surjectivity of $L_1$, hence the existence of solutions. Uniqueness and the upper bound for $M_{p,\kappa,A+B}(0,T)$ follow from \eqref{eq:est Llambda}. 
\end{proof}

\section{Perturbations beyond maximal $L^p$-regularity}\label{sec:linear adm} 

Throughout this section, we assume that $\gamma_*\in (0,1]$ and $X_0$ and $X_{1+\gamma_*}$ are Banach spaces such that $X_{1+\gamma_*}\hookrightarrow X_0$ densely and continuously. 
We assume that  $\{X_\gamma:\gamma\in(0,1+\gamma_*)\}$ is a `reasonable intermediate scale' of Banach spaces, in the sense that
\begin{equation}\label{eq:Xgamma squeeze}
X_{\gamma,1}\into X_\gamma \into X_{\gamma,\infty} \qquad\text{ for all }\gamma\in(0,1+\gamma_*). 
\end{equation}
Here, for $\gamma\in(0,1+\gamma_*)$ and $q\in[1,\infty]$,  $X_{\gamma,q}$ denotes 
the following real interpolation space:
  \begin{equation}\label{eq:def Xgammaq}
      X_{\gamma,q}\ceqq (X_0,X_{1+\gamma_*})_{\frac{\gamma}{(1+\gamma_*)},q}. 
  \end{equation}
  
In particular, the theory in this section can be applied with $X_\gamma$ defined as the complex interpolation space $[X_0,X_{1+\gamma_*}]_{\gamma/(1+\gamma_*)}$ or as the real interpolation space $(X_0,X_{1+\gamma_*})_{\gamma/(1+\gamma_*),q}$ with $q\in[1,\infty]$ or even as the fractional domains of a sectorial operator.

\subsection{A mixed-scale problem} 

For $f\in L^1(0,T;X_0)$, we revisit the problem \eqref{eq:parabolic_problem_prel} and its perturbed version: 
\begin{equation}
\label{eq:parabolic_problem_intro}
\left\{
\begin{aligned}  
u' + (A+B) u= f,\\
u(0)=u_0.
\end{aligned}
\right.
\end{equation}
We consider {\em strong solutions}  as defined  in Subsection \ref{sub: prelims max Lp reg}.
 
In the classical theory of maximal $L^p$-regularity, one considers $f\in L^p(0,T,\wk;X_0)$ and seeks solutions $u$ in the maximal regularity space $\MR^p(0,T,\wk)$ as defined in \eqref{eq:def MRp}. 

In this section, we consider inhomogeneities $f$ belonging to a sum space larger than $L^p(0,T,\wk;X_0)$. Recall that a $k$-tuple of Banach spaces $(Y_1,\ldots,Y_k)$ each continuously embedded into a given Hausdorff topological vector space $\mathcal{H}$ (i.e.\ the {\em ambient space}), 
\[
Y_+ := \plus_{i=1}^k Y_i\ceqq \{y\in \mathcal{H}: y=\textstyle{\sum_{i=1}^{k}}y_i \text{ in }\mathcal{H},\,  y_i\in Y_i \text{ for }i=1,\ldots,k\},
\]
is equipped with the sum norm:
\begin{equation}\label{eq:def sum norm}
  \|y\|_{Y_+}\ceqq \inf\{\textstyle{\sum_{i=1}^{k}}\|y_i\|_{Y_i}: y=\textstyle{\sum_{i=1}^{k}}y_i \text{ in }\mathcal{H},\,  y_i\in Y_i \text{ for }i=1,\ldots,k \}. 
\end{equation} 

We define the sum space for inhomogeneities as follows: 
\begin{align}\label{eq:def sum sp f}  
\x(s,t)\ceqq L^p(s,t,\wk;X_0)+\,{\plus_{i=1}^k}L^{r_i}(s,t,w_{\nu_i};X_{\gamma_i}) 
\end{align} 
with ambient space $L^1(s,t;X_0)$, for suitable parameters $(r_i,\nu_i,\gamma_i)$, which will be specified below. Let us emphasize that we still  write $w_\mu(t)= |t|^\mu$, 
so the weights are not shifted throughout this section. 

We will  consider perturbations of the form $B=\sum_{i=1}^k B_i$, where $B_i\col [0,T]\to \calL(X_1,X_{\gamma_i})$. Note that there is a shift compared to Section \ref{sec:PS extension} (where $B$ maps into $X_0$). 
As noted in the introduction, this framework is motivated by the skeleton equations arising in the weak convergence approach to large deviations for stochastic evolution equations.

To handle such $B$ and $f\in \x(0,T)$, we introduce the mixed-scale maximal regularity space $\MR_+(0,T)$. For $0\leq s<t\leq T$, we define
\begin{align}\label{eq: def MR space}  
\begin{split}
  \MR_+(s,t) \ceqq & \Big(L^p(s,t,\wk; X_1)\cap W^{1,p}(s,t,\wk;X_0)\Big) \\
   &+\, \plus_{i=1}^k \Big( L^{r_i}(s,t,\wni; X_{1+\gamma_i})\cap W^{1,r_i}(s,t,\wni;X_{\gamma_i})\Big),  
\end{split}
\end{align} 
again with ambient space $L^1(s,t;X_0)$.

The aim is to prove well-posedness of \eqref{eq:parabolic_problem_intro} and regularity estimates for   strong solutions $u\in\MR_+(0,T)$, under suitable assumptions on $A$, $B$, and $f$, which will be discussed in Subsection \ref{sub:lin results}. 

We require the parameters $(r_i,\nu_i,\gamma_i)$ appearing in $\x(0,T)$ and $\MR_+(0,T)$ to be \emph{ $(p,\kappa)$-admissible}, which is defined next. In Corollary \ref{cor: emb}, it will be shown that $(p,\kappa)$-admissibility ensures that the solution space $\MR_+(0,T)$  embeds into $L^p(0,T,\wk;X_1)$ and into certain spaces of continuous functions.

\begin{definition}[$(p,\kappa)$-admissibility]\label{def:admissible}
Let $p\in (1,\infty)$ and $\kappa\geq 0$. A triple $(r,\nu,\gamma)$ is $(p,\kappa)$-admissible if $\gamma\in(0,1)\cap(0,\gamma_*]$, $\nu\geq 0$   and 
\begin{align*}
    \nu+1 <r \leq p, \qquad
    \frac{\kappa}{p}\leq \frac{\nu}{r}, \qquad 
  \frac{1+\kappa}{p}<\frac{1+\nu}{r}\leq \frac{1+\kappa}{p}+\gamma.    
\end{align*}
\end{definition}  

\begin{remark}\label{rem:gamma=0 separate}
In the definition above, we require $\gamma\in(0,1)\cap(0,\gamma_*]$, so we can use the spaces $X_{\gamma,q}$ and $X_{1+\gamma,q}$ defined in \eqref{eq:def Xgammaq} (which are e.g.\ essential in the proof of Proposition \ref{prop: emb Lr}). However,  elements of $\x(0,T)$ and $\MR_+(0,T)$ include an $X_0$-valued part (see \eqref{eq:def sum sp f} and \eqref{eq: def MR space}), which will be treated via classical maximal $L^p$-regularity techniques. The techniques for the $X_0$-valued part cannot be used for the $X_\gamma$-valued parts with $\gamma>0$ and vice versa. Consequently, we distinguish between these components in the definitions and proofs.
\end{remark}
\begin{remark}\label{rem:LDP}
In applications to large deviations,   $(p,\kappa)$ is given, and $\gamma=1/2$, and the parameters $(r,\nu)$ of Definition \ref{def:admissible} will be such that $\frac{\kappa}{p}= \frac{\nu}{r}$. Note that in that case the last condition of the definition becomes $\frac{1}{r}\leq \frac{1}{p}+\frac12$. 
\end{remark}

\subsection{Embeddings of the mixed-scale maximal regularity space}

Here, we will prove some embeddings for the space $\MR_+(0,T)$ (defined in \eqref{eq: def MR space}) that  will be essential for establishing our global well-posedness result  for the evolution equation \eqref{eq:parabolic_problem_intro} (Theorem \ref{th: well-posed linear}). 

To begin, we show that the $L^r\cap W^{1,r}$-parts of the space $\MR_+(0,T)$ embed into $L^p(0,T,\wk;X_1)$, under the $(p,\kappa)$-admissibility assumption.

\begin{proposition}\label{prop: emb Lr}
  Let $p\in (1,\infty)$ and let $\kappa\in[0,p-1)$. 
Let $(r, \nu, \gamma)$ be $(p,\kappa)$-admissible.  
Then,   
\begin{equation}\label{eq:emb to show}
\begin{split}
  L^r(0,T&,\wn;X_{1+\gamma})\cap W^{1,r}(0,T,\wn;X_{\gamma}) \into L^{p}(0,T,\wk; X_1).  
  \end{split}
\end{equation} 
The operator norm of the embedding depends on $p,\kappa,r,\nu,\gamma$, and $T$, and blows up if $T\downarrow 0$, unless we assume that the functions vanish at $t=0$. 
\end{proposition}
\begin{proof}  
Firstly, we note that due to \cite[Prop.\ L.4.5]{HNVWvolume3} and by  reflection (see also \cite[Prop.\ 2.5]{AV22nonlinear1}), there exists a bounded extension operator 
\begin{align*}
\mathrm{E}\col L^r(0,T ,\wn;X_{1+\gamma})&\cap W^{1,r}(0,T,\wn;X_{\gamma})\to L^r(\R,\wn;X_{1+\gamma})\cap W^{1,r}(\R,\wn;X_{\gamma}) 
\end{align*}
satisfying $[\mathrm{E}f]|_{[0,T]}=f$. 
By composing with $\mathrm{E}$ and restricting to $[0,T]$ at the very end, it thus suffices to prove the following scheme of embeddings: 
\begin{align}
&\begin{cases}
&W^{1,r}(\R,\wn;X_{\gamma})\into B_{\hat{r}\hat{r}}^{s+1}(\R,w_{\hat{\nu}};X_{\gamma})\\
&L^{r}(\R,\wn;X_{1+\gamma})\into B_{\hat{r}\hat{r}}^{s}(\R,w_{\hat{\nu}};X_{1+\gamma})
\end{cases}\label{eq:emb1}\\[.25em]
&B_{\hat{r}\hat{r}}^{s+1}(\R,w_{\hat{\nu}};X_{\gamma})\cap B_{\hat{r}\hat{r}}^{s}(\R,w_{\hat{\nu}};X_{1+\gamma}) \into B_{\hat{r}\hat{r}}^{\hat{s}}(\R,w_{\hat{\nu}};X_{1})
\label{eq:emb2}\\[.25em]
&B_{\hat{r}\hat{r}}^{\hat{s}}(\R,w_{\hat{\nu}};X_{1})\into L^{p}(\R,\wk; X_1) 
\label{eq:emb3}
\end{align}
for some $s<0$. 
In this embedding scheme, we factorize through Besov spaces to utilize a mixed time-space derivative embedding \eqref{eq:emb2} with microscopic improvement. This allows us to eliminate the real interpolation parameter $\infty$.
Moreover, \eqref{eq:emb1} and \eqref{eq:emb3} will be proved using weighted Sobolev embeddings.
  
Combining \eqref{eq:emb1}--\eqref{eq:emb3} at once yields \eqref{eq:emb to show}, so it remains to show that there exist parameters $s,\hat{s},\hat{r},\hat{\nu}$  such that the embeddings are valid. 

We begin with \eqref{eq:emb1}. First, observe that due to \cite[Prop.\ 3.12]{MV12sharp}:
\begin{align*}
\begin{cases}
&W^{1,r}(\R,\wn;X_{\gamma})\into  F_{r\infty}^1(\R,\wn;X_{\gamma})\\
&L^{r}(\R,\wn;X_{1+\gamma})\into F_{r\infty}^0(\R,\wn;X_{1+\gamma}).
\end{cases}
\end{align*}
Also, we recall that $B_{rr}^s=F_{rr}^s$. 
Therefore, \eqref{eq:emb1} follows from the above combined with the weighted Sobolev embeddings with microscopic improvement for $F$-spaces from  \cite[Th.\ 1.2]{MV12sharp}, imposing for $i=0,1$: 
\begin{equation}\label{eq:emb1cond}
\hat{r}\geq r, \quad \frac{\hat{\nu}}{\hat{r}}\leq \frac{\nu}{r}, \quad \frac{1+\hat{\nu}}{\hat{r}}< \frac{1+\nu}{r}, \quad  -\frac{1+\nu}{r}\geq s-\frac{1+\hat{\nu}}{\hat{r}}.
\end{equation}
This in particular implies that the conditions in \cite[(1.6)]{MV12sharp} are satisfied for the following parameters for $i=0,1$ (the parameters with tilde referring to those in \cite{MV12sharp}): $\tilde{d}=1$, $\tilde{\gamma}_0=\nu>-1$, $\tilde{\gamma}_1=\hat{\nu}>-1$, $\tilde{s}_0=1-i$, $\tilde{s}_1=s+1-i$, $\tilde{p}_0=r$, $\tilde{p}_1=\hat{r}$, $\tilde{q}_0=\infty$, $\tilde{q}_1=\hat{r}$. Thus, if \eqref{eq:emb1cond} is satisfied, then \eqref{eq:emb1} is valid. At the end of this proof, we will show that the range of parameters is non-empty. Note that \eqref{eq:emb1cond} implies that $s\leq 0$. 

Next, for \eqref{eq:emb2} we apply \cite[Th.\ 3.1]{MV14traces} with $\tilde{\theta}=1-\gamma$ and spaces $\tilde{X}_0=X_{\gamma}$, $\tilde{X}_1=X_{1+\gamma}$, $\tilde{X}_{\tilde{\theta}}=X_1$. These 
 spaces satisfy \cite[(3.1)]{MV14traces}, since we have by virtue of \eqref{eq:Xgamma squeeze} and reiteration (see \cite[Th.\ L.3.1]{HNVWvolume3}):
\[
(X_{\gamma},X_{1+\gamma})_{\tilde{\theta},1}\into (X_{\gamma,\infty},X_{1+\gamma,\infty})_{\tilde{\theta},1}= X_{1,1}\into X_1. 
\] 
Moreover, we use the following  parameters for the application of \cite[Th.\ 3.1]{MV14traces}: $\tilde{\alpha}=1$, $\tilde{s}=s$, $\tilde{p}_0=\tilde{q}_0=\hat{r}$, $\tilde{p}_1=\tilde{q}_1=\hat{r}$, $\tilde{p}=\tilde{q}=\hat{r}$, $\tilde{w}_0=w_{\hat{\nu}}$, $\tilde{w}_1=w_{\hat{\nu}}$, $\tilde{w}_\theta=w_{\hat{\nu}}$. These parameters  satisfy the conditions of \cite[Th.\ 3.1]{MV14traces} for the embedding \eqref{eq:emb2} if and only if  
 \begin{equation}\label{eq:emb2cond}
s+\gamma=\hat{s}, 
\end{equation}
where we note that the correct condition for the weights is $\tilde{w}_0^{({1-\tilde{\theta}})/{\tilde{p}_0}}\tilde{w}_1^{{\tilde{\theta}}/{\tilde{p}_1}}
=\tilde{w}_\theta^{{1}/{\tilde{p}}}$, see also \cite[Th.\ 3.12]{ALV23}. 

Finally, we turn to \eqref{eq:emb3}. We use again that $B_{rr}^s=F_{rr}^s$ and apply \cite[Prop.\ 3.12: (3.10)]{MV12sharp} (used for $F_{p,1}^0\into L^p$) and \cite[Th.\ 1.2]{MV12sharp}, imposing that  
 \begin{equation}\label{eq:emb3cond}
p\geq \hat{r}, \quad
\frac{\kappa}{p}\leq \frac{\hat{\nu}}{\hat{r}},\quad \frac{1+\kappa}{p}<\frac{1+\hat{\nu}}{\hat{r}}, 
\quad \hat{s}-\frac{1+\hat{\nu}}{\hat{r}}\geq 0-\frac{1+\kappa}{p}.
\end{equation}

Now,   conditions \eqref{eq:emb1cond}--\eqref{eq:emb3cond} can be rewritten as 
\begin{equation}\label{eq:embconds} 
\begin{cases}
&p\geq \hat{r}\geq r, \\
& \frac{\kappa}{p}\leq \frac{\hat{\nu}}{\hat{r}}\leq \frac{\nu}{r}, \\
&\frac{1+\kappa}{p}<\frac{1+\hat{\nu}}{\hat{r}}<\frac{1+\nu}{r} \\
&s\leq \frac{1+\hat{\nu}}{\hat{r}}-\frac{1+\nu}{r}, \\ 
& s\geq-\gamma+\frac{1+\hat{\nu}}{\hat{r}}  -\frac{1+\kappa}{p},\\
&s+\gamma= \hat{s}. 
\end{cases}
\end{equation}
Combining the lower and upper bounds for $s$, we conclude that 
there   exist $s,\hat{s},\hat{r},\hat{\nu}$ satisfying \eqref{eq:embconds} if and only if:  
\begin{equation}\label{eq:adm for emb}
  p\geq r,\quad\frac{\kappa}{p}\leq \frac{\nu}{r},\qquad \frac{1+\kappa}{p}< \frac{1+\nu}{r},\qquad \frac{1+\nu}{r}\leq \gamma+\frac{1+\kappa}{p}.  
\end{equation}
This last condition is satisfied due to the $(p,\kappa)$-admissibility of $(r,\nu,\gamma)$.  
\end{proof}

Next, we will prove trace embeddings. The embedding into $C_{\nicefrac{\kappa}{p}}((0,T];\Xtrzero)$ plays a crucial role in the proof of our global well-posedness result, Theorem \ref{th: well-posed linear}.

\begin{lemma}\label{lem:trace}  
Let $(r,\nu, \gamma)$ be $(p,\kappa)$-admissible. Then 
\begin{align*}
Z_r& \coloneqq  L^r(0,T,w_{\nu};X_{1+\gamma})\cap W^{1,r}(0,T,w_{\nu};X_{\gamma})
\\ & \hookrightarrow C([0,T];X_{1-\frac{1+\kappa}{p},r})\cap C_{\nicefrac{\kappa}{p}}((0,T];X_{1-\frac{1}{p},r})
\\ & \hookrightarrow C([0,T];\Xtr)\cap C_{\nicefrac{\kappa}{p}}((0,T];\Xtrzero).
\end{align*}
\end{lemma}
\begin{proof}
The last embedding is immediate, since $r\leq p$,
so $X_{\lambda,r}\hookrightarrow X_{\lambda,p}$ for any $\lambda\in (0,1)$.  It remains to
prove the first embedding.  

From \eqref{eq:Xgamma squeeze} and  \cite[Cor.\ L.4.6]{HNVWvolume3} we see that 
\begin{align*}
Z_r&\into L^r(0,T,w_{\nu};X_{1+\gamma,\infty})\cap W^{1,r}(0,T,w_{\nu};X_{\gamma,\infty})\\
&\into  C([0,T];(X_{\gamma,\infty},X_{1+\gamma,\infty} )_{1-\frac{1+\nu}{r},r}).
\end{align*}
This gives the first part of the first embedding, since by reiteration \cite[Th.\ L.3.1]{HNVWvolume3}:
\begin{align*}
(X_{\gamma,\infty},X_{1+\gamma,\infty} )_{1-\frac{1+\nu}{r},r} = X_{1+\gamma-\frac{1+\nu}{r},r} \into  X_{1-\frac{1+\kappa}{p},r}, 
\end{align*}
where we used $(p,\kappa)$-admissibility.  

For the second part of the first embedding, as before, we see that 
\begin{align*}
Z_r \hookrightarrow C_{\nicefrac{\nu}{r}}((0,T];(X_{\gamma,\infty},X_{1+\gamma,\infty} )_{1-\frac{1}{r},r}) = C_{\nicefrac{\nu}{r}}((0,T];X_{1+\gamma-\frac{1}{r},r}).
\end{align*}
This is not the correct spatial smoothness and weight yet. However, we can set  $\theta = \frac{\kappa/p}{\nu/r}$ (convention $0/0 = 0$), which is in $[0,1]$ by $(p,\kappa)$-admissibility, and by interpolation estimates (or reiteration) we find that for all $t\in (0,T]$: 
\begin{align*}
t^{\kappa/p} \|u(t)\|_{X_{1-\frac{1}{p},r}} & \lesssim  \|u(t)\|_{X_{1+\gamma-\frac{1+\nu}{r},r}}^{1-\theta} (t^{\nu/r}\|u(t)\|_{X_{1+\gamma-\frac{1}{r},r}})^{\theta}
\\ & \lesssim \|u\|_{Z_r}^{1-\theta} \|u\|_{Z_r}^{\theta} = \|u\|_{Z_r},
\end{align*}
where we used $(p,\kappa)$-admissibility again, ensuring that $1-\frac1p\leq s\ceqq 1+\gamma-\frac{1+\nu}{r}+\frac{\kappa}{p}$, so $ X_{s,r}=(X_{1+\gamma-\frac{1+\nu}{r},r},X_{1+\gamma-\frac{1}{r},r})_{\theta,r}\into X_{1-\frac1p,r}$ for the case $\theta\in(0,1)$.
\end{proof}

By adding the $L^p\cap W^{1,p}$-part of $\MR_+(0,T)$ to the results above, we obtain the following embedding of the mixed-scale maximal regularity space.

\begin{corollary} \label{cor: emb} 
Let $p\in (1,\infty)$,  $\kappa\in[0,p-1)$ and let $(r_i, \nu_i, \gamma_i)$ be $(p,\kappa)$-admissible for $1\leq i \leq k$. Let $\MR_+(0,T)$ be defined by \eqref{eq: def MR space}. 
Then 
\begin{align*}
  \MR_+(0,T)&\into L^{p}(0,T,\wk; X_1)\cap C([0,T];\Xtr)\cap C_{\nicefrac{\kappa}{p}}((0,T];\Xtrzero). 
  \end{align*}
The operator norm of the embedding depends on $p,\kappa,r_i,\nu_i,\gamma_i$ for $1\leq i \leq k$, and $T$, and blows up if $T\downarrow 0$, unless we assume that the functions vanish at $t=0$.
\end{corollary}

\begin{proof}
Thanks to Proposition \ref{prop: emb Lr} and Lemma \ref{lem:trace}, for each $i\in\{1,\ldots,k\}$, we have a continuous embedding of  
$L^{r_i}(0,T,w_{\nu_i};X_{1+\gamma_i})\cap W^{1,r_i}(0,T,w_{\nu_i};X_{\gamma_i})$ into the intersection of spaces stated in the corollary. 
Recalling \eqref{eq: def MR space}, it remains to show that the same is true for    $L^p(0,T,\wk; X_1)\cap W^{1,p}(0,T,\wk;X_0)$. 
Now, the latter follows  directly  from the trace embeddings in \cite[Cor.\ L.4.6]{HNVWvolume3}. 
\end{proof}

\subsection{Main result}\label{sub:lin results}

Recall that {\em strong solutions}  were defined   in Subsection \ref{sub: prelims max Lp reg}.  
The aim of this section is to prove global well-posedness and regularity estimates for   strong solutions to \eqref{eq:parabolic_problem_intro}. These properties are all contained in our main result, Theorem \ref{th: well-posed linear} below. 
We now present the assumptions on the coefficients $A$ and $B$. 

\begin{assumption}\label{ass:linear A transference}
$\gamma_*>0$, $T\in (0,\infty)$, $p\in(1, \infty)$, $\kappa\in[0,p-1)$, and the following holds:
\begin{enumerate}[\rm (1)]
  \item\label{it:A1} $X_0$ and $X_{1+\gamma_*}$ are Banach spaces such that $X_{1+\gamma_*}\hookrightarrow X_0$ densely and continuously.  
For $\gamma\in(0,{1+\gamma_*})$, $X_\gamma$ is a Banach space such that
\[
X_{\gamma,1}\into X_\gamma \into X_{\gamma,\infty},
\]
where $X_{\gamma,q}\ceqq (X_0,X_{1+\gamma_*})_{\frac{\gamma}{{1+\gamma_*}},q}$ for $q\in[1,\infty]$. 
  \item\label{it:A2} $A\col[0,T]\to \calL(X_1, X_0)$ is strongly measurable in the strong operator topology and $A$ has maximal $L^p(0,T,\wk;X_0)$-regularity (see Definition \ref{def:MR}). 
  \item\label{it:A3} For every $(p,\kappa)$-admissible triple $(r, \nu, \gamma)$ (see Definition \ref{def:admissible}), there exists a map  $A_{r,\nu,\gamma}\col [0,T]\to \calL(X_{1},X_{0})$ that is strongly measurable in the strong operator topology, and $A_{r,\nu,\gamma}|_{X_{\gamma+1}}\col [0,T]\to \calL(X_{\gamma+1},X_{\gamma})$ is well-defined and has maximal $L^r(0,T,w_{\nu};X_\gamma)$-regularity. 
\end{enumerate}
Moreover,   the following mapping property holds for $A$ and $A_{r,\nu,\gamma}$,  for every $(p,\kappa)$-admissible triple $(r, \nu, \gamma)$:  there exist constants $C_A$ and $C_{A_{r,\nu,\gamma}}$ such that for all $u\in L^p(0,T,\wk;X_1)\cap C([0,T];\Xtr)\cap C_{\nicefrac{\kappa}{p}}((0,T];\Xtrzero)$, it holds that 
\begin{equation}\label{eq:A mapping sum}
\begin{cases}
      \|Au\|_{L^p(0,T,\wk;X_0)} \leq C_A\Big(\|u\|_{L^p(0,T,\wk;X_1)}+\|u\|_{C([0,T];\Xtr)}+\|u\|_{C_{\nicefrac{\kappa}{p}}((0,T];\Xtrzero)} \Big) \\
      \|A_{r,\nu,\gamma}u\|_{L^p(0,T,\wk;X_0)} \leq 
      C_{A_{r,\nu,\gamma}}\Big(\|u\|_{L^p(0,T,\wk;X_1)}+\|u\|_{C([0,T];\Xtr)}+\|u\|_{C_{\nicefrac{\kappa}{p}}((0,T];\Xtrzero)} \Big).
\end{cases}
\end{equation}  
\end{assumption}

Since there are no derivatives in \eqref{eq:A mapping sum}, one can show that \eqref{eq:A mapping sum} extends to all subintervals $[s,t]$ with the same constant $C_A$ (keeping the same unshifted weights). Indeed, this can be shown by building an extension operator with norm $\leq 1+\varepsilon$ using compression, reflection and a smooth cut-off. Since $\varepsilon>0$ can be arbitrary small, this gives the required result. 

We note that the mapping property \eqref{eq:A mapping sum} is similar to that in  \eqref{eq:bddcondLpA} used in Section \ref{sec:PS extension}. The  version  \eqref{eq:A mapping sum} will be sufficient to obtain   uniqueness in the current setting. Other variants of this condition are also possible: see Remark \ref{rem:Amappingchoice}.   

\begin{remark}
In Assumption \ref{ass:linear A transference}\textit{(\ref{it:A1})},  $X_\gamma$ can always be chosen as a complex interpolation space $[X_0,X_{1+\gamma_*}]_{\frac{\gamma}{1+\gamma_*}}$ or a real interpolation space  $(X_0,X_{1+\gamma_*})_{\frac{\gamma}{1+\gamma_*},q}$ with $q\in[1,\infty]$. 

Since $X_{1,1}\into X_1\into X_{1,\infty}$, one has 
\[
X_{\theta,q}=(X_0,X_1)_{\theta,q}, \qquad \theta\in(0,1), q\in[1,\infty]
\]
with equivalent norms, by reiteration (see \cite[Th.\ L.3.1]{HNVWvolume3}). 

Finally, we note that the mappings $A_{r,\nu,\gamma}$ in Assumption \ref{ass:linear A transference}\textit{(\ref{it:A3})} may be chosen freely, and need not be related to $A$. For instance $A_{r,\nu,\gamma}$ could be chosen time-independent. One could also choose $A_{r,\nu,\gamma}\ceqq A|_{X_{\gamma+1}}$ (e.g.\ in the setting of Corollary \ref{cor:time-indep}). Note that  $A|_{X_{\gamma+1}}(\cdot)x$ is defined for $x\in X_{1+\gamma}$, since $X_{\gamma+1}\into X_{\gamma+1,\infty}\into X_{1,1}\into X_1$.  
\end{remark}

We   emphasize that for time-independent $A$, Assumption \ref{ass:linear A transference} can be weakened substantially. It suffices to require only maximal $L^p(0,T;X_0)$-regularity of $A$, and allowed choices for $X_\gamma$ (for $\gamma\in(0,1)$) are complex and real  interpolation spaces for the pair $(X_0,D(A))$  and fractional domains $D(A^\gamma)$. 
The details will be given in Corollary \ref{cor:time-indep}.

Next, we turn to $B$. In the assumption below,  we fix a finite family of $(p,\kappa)$-admissible triples  and require a suitable integrability condition. 

\begin{assumption}\label{ass: B linear}
$p\in(1,\infty)$, $\kappa\geq 0$, $k\in\N$ and  $(r_i,\nu_i,\gamma_i)$ is $(p,\kappa)$-admissible for $i=1,\ldots,k$. 
Moreover, $B\col [0,T]\to \calL(X_1,X_0)$ decomposes as $B=\sum_{i=1}^k B_i$, where each $B_i\col [0,T]\to \calL(X_1,X_{\gamma_i})$ is strongly measurable in the strong operator topology  and 
  \[
  \|B_i(\cdot)\|_{\calL(X_1,X_{\gamma_i})}\leq b_i(\cdot),\, \text{ and }\, 
  \begin{cases}b_i\in L^{\frac{pr_i}{p-r_i}}(0,T,w_{\frac{\nu_i p-\kappa r_i}{p-r_i}}), \,  &\text{if }p>r_i,\\
  b_i=0, &\text{if }p=r_i.  
  \end{cases}
  \]  
\end{assumption}

\begin{remark}
In Assumption \ref{ass: B linear}, $B_i=0$ whenever $p=r_i$. However, admissible triples $(r_i,\nu_i,\gamma_i)$ with $p=r_i$ are retained to accommodate additional inhomogeneities $f_i$ later. Also, note that if $p>r_i$, then $(p,\kappa)$-admissibility implies that $\frac{\nu_i p-\kappa r_i}{p-r_i} \geq 0$  for the weight of $b_i$. In applications to large deviations (see Remark \ref{rem:LDP}), the weight in Assumption \ref{ass: B linear} becomes one. In this case, $b_i\in L^{q_i}(0,T)$ with $\frac{1}{q_i} = \frac{1}{r_i}-\frac{1}{p}$. 
\end{remark} 

From now on, whenever Assumption \ref{ass: B linear} holds, we use the following notation for convenience: 
 \begin{equation}\label{eq: def b_T}
b_{\scriptscriptstyle[0,\tau]}\ceqq \sum_{1\leq i\leq k, b_i\neq 0}  \|b_i\|_{ L^{\frac{pr_i}{p-r_i}}(0,\tau,w_{\frac{\nu_i p-\kappa r_i}{p-r_i}})}.
 \end{equation}

Moreover, if  $A$ and $A_{r,\nu,\gamma}$ are as in Assumption \ref{ass:linear A transference}, then we will write their maximal regularity constants as $M_{p,\kappa,A}(0,T)$ and $M_{r,\nu,A_{r,\nu,\gamma}}(0,T)$, in line with Definition \ref{def:MR}. Under Assumption \ref{ass: B linear} (so that admissible triples are fixed), we let 
 \begin{equation}\label{eq: def M_T}
   M_T\ceqq M_{p,\kappa,A}(0,T) \vee \max_{1\leq i\leq k} M_{r_i,\nu_i,A_{r_i,\nu_i,\gamma_i}}(0,T).
 \end{equation}

The following is our main result. 

\begin{theorem}\label{th: well-posed linear}
Let Assumptions \ref{ass:linear A transference} and \ref{ass: B linear} hold. Let  $\x(0,T)$ and $\MR_+(0,T)$ be defined by \eqref{eq:def sum sp f} and \eqref{eq: def MR space}. 
Let $u_0\in (X_0, X_1)_{1-\frac{1+\kappa}{p},p}$  and $f\in \x(0,T)$. 

Then, \eqref{eq:parabolic_problem_intro} has a unique strong solution $u\in\MR_+(0,T)$ and
\begin{equation*}
\begin{split}
\|u\|_{\MR_+(0,T)} +&\|u\|_{L^{p}(0,T,\wk;X_{1})}+\|u\|_{C([0,T];X_{1-\frac{1+\kappa}{p},p})} +\|u\|_{C_{\nicefrac{\kappa}{p}}((0,T];X_{1-\frac1p,p})}\\
 & \leq C \big(\|u_0\|_{X_{1-\frac{1+\kappa}{p},p}} +  \|f\|_{\x(0,T)}\big),  
\end{split}
\end{equation*}
for a constant $C$ independent of $u_0$ and $f$. 
The constant $C$ depends only on  $p$, $\kappa$, and     $r_i,\nu_i,\gamma_i$ for $1\leq i\leq k$, and nondecreasingly on $b_{\scriptscriptstyle[0,T]}$, $M_T$ defined by \eqref{eq: def M_T}, and $C_A$ and  $C_{A_{r_i,\nu_i,\gamma_i}}$  from \eqref{eq:A mapping sum}. 
\end{theorem}

This result is of interest even in the unperturbed case ($B=0$). Indeed, we extend the space of inhomogeneities to $\x(0,T)\supset L^p(0,T,\wk;X_0)$, while preserving the regularity $u\in L^p(0,T,\wk;X_1)$.

 \begin{remark}\label{rem: wlog non decr in T}
 It readily follows from Theorem \ref{th: well-posed linear} that the claims therein remain true on subintervals $[0,\tau]$ with $\tau\in(0,T]$, with the same constant $C$ (i.e.\ for the case of the full interval $[0,T]$). 
 
 Existence and uniqueness of solutions on $[0,\tau]$ hold since  Assumptions \ref{ass:linear A transference} and \ref{ass: B linear} are also satisfied on $[0,\tau]$, due to Lemma \ref{lem:restriction} and since \eqref{eq:A mapping sum} holds on $[0,\tau]$ (with new constants) by applying a bounded  extension operator to $u$. 
 
 Concerning the final constant $C$, we can apply   
Theorem \ref{th: well-posed linear} to $f\one_{[0,\tau]}\in \x(0,T)$, restrict the corresponding solution (on $[0,T]$) to $[0,\tau]$,  and use uniqueness on $[0,\tau]$.  
 \end{remark} 
 
 When $A$ is a (time-independent) linear operator, the main theorem simplifies significantly, rendering several assumptions redundant.
 \begin{corollary}[Autonomous case] \label{cor:time-indep}
   Let $A$ be a densely defined linear operator on $X_0$, let $X_1\ceqq D(A)$ and suppose that $A$ has maximal $L^p(0,T;X_0)$-regularity.    Suppose that we have one of the following definitions for $X_\gamma$:
   \begin{align} 
     &X_\gamma\ceqq\begin{cases}
       [X_0,X_1]_{\gamma},\qquad & \gamma\in (0,1),  \\
       [X_1,D(A^2)]_{\gamma-1}, &\gamma\in (1,2), 
     \end{cases}\label{eq:complex}\\[.3em]
     &X_\gamma\ceqq\begin{cases}
       (X_0,X_1)_{\gamma,q},\qquad & \gamma\in (0,1),  \\
       (X_0,D(A^2))_{\gamma/2,q}, &\gamma\in (1,2), 
     \end{cases}\label{eq:real}\\[.5em]
     &X_\gamma\ceqq 
       D((\mu+A)^\gamma), \qquad\, \gamma\in(0,1)\cup(1,2), 
       \label{eq:fractional}
   \end{align}
   where $q\in[1,\infty]$ in \eqref{eq:real} and $\mu\in \R$ is fixed and large enough.  Let Assumption \ref{ass: B linear} hold. 

Then the conclusions of Theorem \ref{th: well-posed linear} are valid.
 \end{corollary}
 Note that maximal $L^p$-regularity implies sectoriality of a shifted version of $A$ (see \cite[Th.\ 17.2.15]{HNVWvolume3}). In particular, $A$ is closed and generates an analytic $C_0$-semigroup. Moreover, we can find $\mu\in \R$ such that $\mu+A$ is sectorial and invertible.  
 
\begin{proof} 
By applying the invertible transformation $u\mapsto e^{-\lambda \cdot} u$ on $\MR_+(0,T)$, we can equivalently consider the problem with $A$ replaced by $\lambda+A$. Therefore, in the sequel, without loss of generality, we can assume $A$ is sectorial and invertible and $\mu=0$ in \eqref{eq:fractional}. Since $A$ is densely defined and sectorial, it holds that $X_2\into X_0$ densely (see \cite[Prop.\ 2.1.1(i), Prop.\ 2.1.4(i)]{Lun}).

Put $\gamma_*\ceqq 1$,  $X_2\ceqq D(A^2)$ and write $X_{\gamma,q}\ceqq (X_0,D(A^2))_{\gamma/2,q}$ for $\gamma\in(0,2)$ and $q\in[1,\infty]$. Define $A_{r,\nu,\gamma}\ceqq A|_{X_{\gamma+1}}$ whenever  $(r,\nu,\gamma)$ is a $(p,\kappa)$-admissible triple.  
 
To prove the corollary, it remains to prove that Assumption \ref{ass:linear A transference} is satisfied for all three choices of $X_\gamma$, and with $A_{r,\nu,\gamma}$ defined above.  
Note that  \eqref{eq:A mapping sum} holds since $X_1=D(A)$, so  $\|Au\|_{L^p(0,T,\wk;X_0)}\lesssim_{A,X_1,X_0} \|u\|_{L^p(0,T,\wk;X_1)}$ for any $u\in L^p(0,T,\wk;X_1)$. 
Moreover, we can validate Assumption \ref{ass:linear A transference}\textit{(\ref{it:A2})} at once: maximal $L^p(0,T;X_0)$-regularity implies maximal $L^p(0,T,\wk;X_0)$-regularity by \cite[Prop.\ 17.2.36]{HNVWvolume3} (note that $\kappa/p\in [0,1-1/p)$ by Assumption \ref{ass: B linear}). 
It remains to verify Assumption \ref{ass:linear A transference}\textit{(\ref{it:A1})}\textit{(\ref{it:A3})}. We   do this case by case below. One more general observation will be useful:  $(X_0, D(A^2))_{\frac12, 1} \hookrightarrow D(A)\hookrightarrow (X_0, D(A^2))_{\frac12, \infty}$ by \cite[\S1.15.2 Th.(d)]{Tri78}. Therefore, using reiteration,  we obtain 
\begin{align}\label{eq:D(A)}
  (X_0,D(A))_{\gamma,q}& =(X_0,D(A^2))_{\gamma/2,q}\eqqc X_{\gamma,q}, \qquad \gamma\in(0,1),\, q\in[1,\infty],
\\ \label{eq:D(A^2)}
  (D(A),D(A^2))_{\gamma-1,q}&=(X_0,D(A^2))_{\gamma/2,q}\eqqc X_{\gamma,q}, \qquad \gamma\in(1,2),\, q\in[1,\infty].
\end{align}
 
\textbf{Case of \eqref{eq:complex}: }   
We have by \eqref{eq:D(A)}, \eqref{eq:D(A^2)} and \cite[\S1.10.3 Th.\ 1]{Tri78}:  
\[
\begin{cases}
X_{\gamma,1}\into [X_0,D(A)]_\gamma \into X_{\gamma,\infty}, \qquad &\gamma\in(0,1),\\
X_{\gamma,1}\into  [D(A),D(A^2)]_{\gamma-1} \into  X_{\gamma,\infty}, & \gamma\in(1,2),
\end{cases}
\]
thus $X_\gamma$ defined by \eqref{eq:complex} satisfies Assumption \ref{ass:linear A transference}\textit{(\ref{it:A1})}.

To verify Assumption \ref{ass:linear A transference}\textit{(\ref{it:A3})}, it suffices to prove maximal $L^p(0,T;X_\gamma)$-regularity, since maximal $L^r(0,T,\wn;X_\gamma)$-regularity then follows by applying \cite[Th.\ 17.2.31, Prop.\ 17.2.36]{HNVWvolume3} (note that  $\nu/r\in [0,1-1/r)$ by Definition \ref{def:admissible}). 

We prove maximal $L^p(0,T;X_\gamma)$-regularity. 
Maximal $L^p(0,T;X_0)$-regularity implies maximal $L^p(0,T;X_1)$-regularity, using that  
$A\col X_{i+1}\to X_i$ is an isomorphism for $i=0,1$ and applying $A$ to the problem \eqref{eq:parabolic_problem_prel} with $u_0=0$. 

Hence, for $i=0,1$, the map $\Phi_i\col L^p(0,T;X_i)\to L^p(0,T;X_{i+1}) \col  f\mapsto u_f$ is well-defined, linear and bounded, where $u_f\in L^p(0,T;X_1)\cap W^{1,p}(0,T;X_0)$ is the unique strong solution to $u' + Au= f$, $u(0)=0$. 
Exactness of complex interpolation gives that $\Phi\col f\mapsto u_f\col L^p(0,T;[X_0,X_1]_\gamma)\to L^p(0,T;[X_1,X_2]_\gamma)$ is bounded for $\gamma\in(0,1)$.  

Now, a   consequence  is that  $f\mapsto u_f\col L^p(0,T;[X_0,X_1]_\gamma)\to W^{1,p}(0,T;[X_0,X_1]_\gamma)$ is also bounded,   since $u_f'=-Au_f+f$ a.e.\ (see \cite[Lem.\ 2.5.8]{HNVWvolume1}), combined with boundedness of $\Phi$ and the fact that 
\begin{equation}\label{eq:iso interpol}
  A\col [X_1,X_2]_\gamma\to [X_0,X_1]_\gamma \text{ is an isomorphism, }
\end{equation}
which follows from exactness of complex interpolation applied to $A$ and $A^{-1}$. 
Also, \eqref{eq:iso interpol}   yields $D(A|_{X_\gamma})=X_{1+\gamma}$. 
This proves the needed existence and maximal regularity estimate. Uniqueness follows from maximal $L^p(0,T;X_0)$-regularity since $X_{i+\gamma}\subset X_i$ for $i=0,1$,  $\gamma\in(0,1)$.  

\textbf{Case of \eqref{eq:real}: } 
Recalling \eqref{eq:real} and using \eqref{eq:D(A)}, we have $X_\gamma=X_{\gamma,q}$ for all $\gamma\in(0,2)$ and the embeddings of Assumption \ref{ass:linear A transference}\textit{(\ref{it:A1})} hold. 

Regarding Assumption \ref{ass:linear A transference}\textit{(\ref{it:A3})}, $A|_{X_{\gamma}}$ has maximal $L^r(0,T;X_{\gamma})$-regularity 
by \cite[Cor.\ 17.3.20]{HNVWvolume3} (the proof therein treats $q\in[1,\infty]$). Now  \cite[Prop.\ 17.2.36]{HNVWvolume3} yields  maximal $L^r(0,T,\wn;X_{\gamma})$-regularity of $A|_{X_{\gamma}}$. 
Here we use that $D(A|_{X_{\gamma}})=X_{1+\gamma,q}\eqqc X_{1+\gamma}$, which follows from \cite[\S1.15.2 Th.(e)]{Tri78} (or by using   exactness as in \eqref{eq:iso interpol}).  
   
\textbf{Case of \eqref{eq:fractional}: } 
For $\gamma\in(0,2)$, we have $X_{\gamma,1}\into D(A^\gamma)\into X_{\gamma,\infty}$  by \cite[\S1.15.2 Th.(d)]{Tri78}, and  $A^\gamma\col X_\gamma \to X_0$ is a Banach space isomorphism by \cite[\S1.15.2 Th.(e)]{Tri78}. 
 
Now, maximal $L^p(0,T;X_0)$-regularity implies maximal $L^r(0,T;X_0)$-regularity \cite[Th.\ 17.2.26(4)]{HNVWvolume3}, which implies maximal $L^r(0,T,\wn;X_0)$-regularity \cite[Prop.\ 17.2.36]{HNVWvolume3}. 
The latter implies maximal $L^r(0,T,\wn;X_\gamma)$-regularity, as can be seen from transforming the problem \eqref{eq:parabolic_problem_prel} using the isomorphism $A^\gamma$, and noting that $D(A|_{D(A^\gamma)})=D(A^{1+\gamma})$ \cite[\S1.15.2 Th.(e)]{Tri78}. 
 \end{proof}

\subsection{Proof of Theorem \ref{th: well-posed linear}}

We return to  the main theorem, Theorem \ref{th: well-posed linear}, in which $A$ is time-dependent.  Before proceeding to the proof, we need several preparations. 
The next  lemma establishes  estimates that  Assumption \ref{ass:linear A transference} gives for  
\begin{equation}\label{eq: S*f}
\left\{
\begin{aligned}
    u' + A u &= f, \\
    u(0) &= 0.
\end{aligned}
\right.
\end{equation} 
We use the notation of the maximal regularity constants from Definition \ref{def:MR}.  Note that uniqueness for  part (ii) of the lemma below will be proved more generally in Proposition \ref{prop: linearized well-posed}.  

\begin{lemma}\label{lem:Sf Lr MR transference} 
Let Assumption \ref{ass:linear A transference} hold and let $(r,\nu,\gamma)$ be $(p,\kappa)$-admissible.  
Then, for all $\tau\in(0,T]$, we have:
\begin{itemize}
\item[(i)] For every $f\in L^p(0,\tau,\wk;X_0)$, \eqref{eq: S*f} has a unique strong solution $u\in L^p(0,\tau,\wk;X_1)\cap W^{1,p}(0,\tau,\wk;X_0)$ on $[0,\tau]$, and  
\begin{equation*}
\|u\|_{L^p(0,\tau,\wk;X_1)}+\|u\|_{W^{1,p}(0,\tau,\wk;X_0)}\leq M_{p,\kappa,A}(0,T) \|f\|_{L^p(0,\tau,\wk ;X_0)}. 
\end{equation*} 
  \item[(ii)] For every $f\in L^r(0,\tau,\wn ;X_{\gamma})$, \eqref{eq: S*f} has a strong solution $u\in F(0,\tau)$ such that 
\begin{equation*}
\|u\|_{F(0,\tau)}\leq M_{r,\nu,A_{r,\nu,\gamma}}(0,T)(1+M_{p,\kappa,A}(0,T)L(C_{A_{r,\nu,\gamma}}+C_A )) \|f\|_{L^{r}(0,\tau,\wn ;X_{\gamma})}, 
\end{equation*}
where $L$ is a constant depending only on $p,\kappa,r,\nu,$ and $\gamma$, and $C_A$ and $C_{A_{r,\nu,\gamma}}$ are the constants from  \eqref{eq:A mapping sum},
and 
\[F(0,\tau) \ceqq \big(L^p(0,\tau,\wk;X_1)\cap W^{1,p}(0,\tau,\wk;X_0)\big)+\big(L^r(0,\tau,\wn;X_{\gamma+1})\cap W^{1,r}(0,\tau,\wn;X_{\gamma})\big).\] 
\end{itemize} 
\end{lemma}
\begin{proof}
(i): By the maximal $L^p(0,T,\wk;X_0)$-regularity from Assumption \ref{ass:linear A transference}\textit{(\ref{it:A2})} and by Lemma \ref{lem:restriction}, we have a unique strong solution in $L^p(0,\tau,\wk;X_1)\cap W^{1,p}(0,\tau,\wk;X_0)$ to \eqref{eq: S*f}. The claimed estimate follows from \eqref{eq:MRest}. 

(ii): By the maximal $L^r(0,T,\wn;X_\gamma)$-regularity in  Assumption \ref{ass:linear A transference}\textit{(\ref{it:A3})}, there exists   a unique strong solution $v\in L^r(0,\tau,\wn;X_{1+\gamma})\cap W^{1,r}(0,\tau,\wn;X_{\gamma})$  to $v'+A_{r,\nu,\gamma}v=f$, $v(0)=0$. Moreover, using a bounded extension operator $\mathrm{E}\col L^r(0,\tau,\wn;X_{\gamma+1})\cap W^{1,r}(0,\tau,\wn;X_{\gamma})\to L^r(0,\tau,\wn;X_{\gamma+1})\cap W^{1,r}(0,T,\wn;X_{\gamma})$ (see \cite[Prop.\ L.4.5]{HNVWvolume3}) and using  \eqref{eq:A mapping sum}, we have
\begin{align}\label{eq:trans est}
\|(A_{r,\nu,\gamma}-A)v\|_{L^p(0,\tau,\wk;X_0)}&\leq \|(A_{r,\nu,\gamma}-A)\mathrm{E}v\|_{L^p(0,T,\wk;X_0)}\notag\\&
\leq (C_{A_{r,\nu,\gamma}}+C_A )K\|\mathrm{E}\|\|v\|_{L^r(0,\tau,\wn;X_{\gamma+1})\cap W^{1,r}(0,\tau,\wn;X_{\gamma})},
\end{align}
where $K$ is the embedding constant of Corollary \ref{cor: emb}, which is $\tau$- and $T$-independent since $v(0)=0$, and $\|\mathrm{E}\|$ depends only on $p,\kappa,r,$ and $\nu$  (see \cite[Prop.\ L.4.5(4)]{HNVWvolume3}).  
Therefore, by Assumption \ref{ass:linear A transference}\textit{(\ref{it:A2})}, we have a unique strong solution $w\in L^p(0,\tau,\wk;X_1)\cap W^{1,p}(0,\tau,\wk;X_0)$ to $w'+Aw=(A_{r,\nu,\gamma}-A)v$, $w(0)=0$. 

Now put $u\ceqq v+w$. Then $u\in F(0,\tau)$ and $u$ is a strong solution to \eqref{eq: S*f}. Moreover, by maximal regularity,  Lemma \ref{lem:restriction}, and \eqref{eq:trans est},  
\begin{align*}
\|u\|_{F(0,\tau)}&\leq \|v\|_{L^r(0,\tau,\wn;X_{\gamma+1})\cap W^{1,r}(0,\tau,\wn;X_{\gamma})}+\|w\|_{L^p(0,\tau,\wk;X_1)\cap W^{1,p}(0,\tau,\wk;X_0)}\\
&\leq \big(1+M_{p,\kappa,A}(0,T)(C_{A_{r,\nu,\gamma}}+C_A )K\|\mathrm{E}\|\big)\|v\|_{L^r(0,\tau,\wn;X_{\gamma+1})\cap W^{1,r}(0,\tau,\wn;X_{\gamma})}\\
&\leq M_{r,\nu,A_{r,\nu,\gamma}}(0,T)\big(1+M_{p,\kappa,A}(0,T)(C_{A_{r,\nu,\gamma}}+C_A )K\|\mathrm{E}\|\big)\|f\|_{L^r(0,\tau,\wn;X_{\gamma})}.
\end{align*} 
\end{proof}

The following lemma is an extension of Lemma \ref{lem:initial}, in which we assume that $A$ has the mapping property \eqref{eq:A mapping sum} instead of \eqref{eq:bddcondLpA}. We consider the homogeneous problem \begin{equation}\label{eq: in value part}
\begin{cases}
{u}'+A{u}=0,\\
{u}(0)=u_0. 
\end{cases}
\end{equation}

\begin{lemma}\label{lem:initial mixed scale}
Let Assumption \ref{assump:XA}  be satisfied.
Let $p\in (1, \infty)$ and $\kappa\in [0,p-1)$.  Suppose that  the condition for $A$ in \eqref{eq:A mapping sum} holds  and $A$ has maximal $L^p(0,T,w_{\kappa};X_0)$-regularity. 

For every $\tau\in(0,T]$ and $u_0\in X_{1-\frac{1+\kappa}{p},p}$, there exists a unique strong solution $u\in L^p(0,\tau,\wk;X_1)\cap W^{1,p}(0,\tau,\wk;X_0)$ to \eqref{eq: in value part}, and  
\begin{align*}
&\|u\|_{L^p(0,\tau,\wk;X_1)\cap W^{1,p}(0,\tau,\wk;X_0)}+\|u\|_{C([0,\tau];\Xtr)} +\|u \|_{C_{\nicefrac{\kappa}{p}}((0,\tau];\Xtrzero)} \leq R_T \|u_0\|_{X_{1-\frac{1+\kappa}{p},p}},  
\end{align*}
where $R_T$ is a constant that depends only on $p,\kappa$, and nondecreasingly on $M_{p,\kappa,A}(0,T)$ and $C_A$. In particular, $R_T$ is independent of $\tau$ and $u_0$.  
\end{lemma} 
\begin{proof}
Uniqueness follows from \eqref{eq:MRest}. For the existence and estimate, we exploit the trace method for real interpolation (see \cite[{\S}L.2]{HNVWvolume3} and \cite[\S 1.8]{Tri78}). Using the definition of the  $\MR^p$-space in \eqref{eq:def MRp}, let $z\in \MR^p(0,\infty,w_{\kappa})$ be such  that $z(0) = u_0$. 
Note that $z',Az\in L^p(0,\tau,\wk;X_0)$ thanks to \eqref{eq:A mapping sum}, \eqref{eq:traceembedding0T} and \eqref{eq:traceembeddingaT}. 
Let $v\in \MR^p(0,\tau,w_{\kappa})$ be the unique solution to the problem $v' + A v= - z'-Az$ with $v(0) = 0$. 
Then, $u \ceqq v+z$ satisfies \eqref{eq: in value part}. It remains to prove the estimates for $u$. 

Write $M\ceqq M_{p,\kappa,A}(0,T)$. 
By maximal regularity and by \eqref{eq:A mapping sum}, \eqref{eq:traceembedding0T} and \eqref{eq:traceembeddingaT}, we have
\begin{align}
 \|v\|_{\MR^p(0,\tau,w_{\kappa})} &\leq M \big(\|z'\|_{L^p(0,T,w_{\kappa};X_0)}+ \|A z\|_{L^p(0,T,w_{\kappa};X_0)}\big) \notag\\
&\leq M \|z'\|_{L^p(0,T,w_{\kappa};X_0)}\notag\\
&\quad+MC_A\big(\|z\|_{L^p(0,T,\wk;X_1)}+\|z\|_{C([0,T];\Xtr)}+\|z\|_{C_{\nicefrac{\kappa}{p}}((0,T];\Xtrzero)} \big) 
\notag\\ & \leq  M(1+ C_A) \|z\|_{\MR^p(0,\infty,w_{\kappa})}+MC_A\big(\|z\|_{C([0,\infty);\Xtr)}+\|z\|_{C_{\nicefrac{\kappa}{p}}((0,\infty);\Xtrzero)} \big)\notag\\
& \leq  \big(M(1+ C_A)+MC_A(K_{\kappa,p}+K_{0,p})\big) \|z\|_{\MR^p(0,\infty,w_{\kappa})}, \label{eq:v est}
\end{align} 
where  $K_{\kappa,p}\ceqq 64p(\frac{1}{p-1-\kappa}\vee\frac{1}{1+\kappa})$ and  $K_{0,p}\ceqq 64p(\frac{1}{p-1}\vee1)$ are constants independent of $\tau$ and $T$, since \eqref{eq:traceembedding0T} and \eqref{eq:traceembeddingaT} are applied on $[0,\infty)$ in the last line (see \cite[Th.\ L.4.1]{HNVWvolume3} for the constants). 

We conclude that
\begin{align*}
\|u\|_{\MR^p(0,\tau,w_{\kappa})} &\leq \|v\|_{\MR^p(0,\tau,w_{\kappa})}+ \|z\|_{\MR^p(0,\infty,w_{\kappa})} \\
& \leq  \big(M(1+ C_A)+1+MC_A(K_{\kappa,p}+K_{0,p})\big) \|z\|_{\MR^p(0,\infty,w_{\kappa})}.
\end{align*}
Taking the infimum over all $z\in \MR^p(0,\infty,w_{\kappa})$, the first stated  estimate  follows from the trace method for real interpolation (with an additional  constant $\frac{12p}{1+\kappa}$, see \cite[Th.\ L.2.3]{HNVWvolume3}). 

Moreover, by \eqref{eq:traceembedding0T},
\begin{align*}
\|u\|_{C([0,\tau];\Xtr)} &\leq \|v\|_{C([0,\tau];\Xtr)}+ \|z\|_{C([0,\infty);\Xtr)} \\
&\leq  K_{\kappa,p}C_{\kappa,p}\|v\|_{\MR^p(0,\tau,w_{\kappa})}+ K_{\kappa,p}\|z\|_{\MR^p(0,\infty,w_{\kappa})}, 
\end{align*} 
where $C_{\kappa,p}\ceqq 2+3(1-\frac{\kappa}{p-1})^{\frac{p-1}{p}}$ is a constant that is also independent of $\tau$ and $T$ since $v(0)=0$, see \cite[Th.\ L.4.1, Prop.\ L.4.5]{HNVWvolume3}. 

Similarly, by \eqref{eq:traceembeddingaT},
\begin{align*}
\|u \|_{C_{\nicefrac{\kappa}{p}}((0,\tau];\Xtrzero)} 
&\leq \|v\|_{C_{\nicefrac{\kappa}{p}}((0,\tau];\Xtrzero)}+ 
\|z\|_{C_{\nicefrac{\kappa}{p}}((0,\infty);\Xtrzero)} \\
&\leq  K_{0,p}C_{\kappa,p}\|v\|_{\MR^p(0,\tau,w_{\kappa})}+ K_{0,p}\|z\|_{\MR^p(0,\infty,w_{\kappa})}, 
\end{align*} 
see again \cite[Th.\ L.4.1, Prop.\ L.4.5]{HNVWvolume3} and recall that $v(0)=0$.  Combining these estimates with \eqref{eq:v est}, and taking the infimum over all $z\in \MR^p(0,\infty,w_{\kappa})$, the trace method for real interpolation completes the proof of the last two estimates (again with extra constant $\frac{12p}{1+\kappa}$, see \cite[Th.\ L.2.3]{HNVWvolume3}).
\end{proof}

In the sequel we will often use a weighted H\"older inequality, which we state explicitly for convenience. For $1<r< p,q<\infty$ such that $\frac{1}{r} = \frac{1}{p} + \frac1q$
and $\kappa,\nu\in\R$ one has that  
\begin{equation}\label{eq:weightedHolder}
 \|fg\|_{L^r(0,T,w_\nu)}\leq 
 \|f\|_{L^{q}(0,T,w_{\frac{\nu p-\kappa r}{p-r}})}
 \|g\|_{L^p(0,T,\wk)}, \ \ \ \ f,g\in L^0(0,T). 
\end{equation}

Combining earlier results, we obtain a useful estimate for the $B$-term in \eqref{eq:ABfu_0}. 
\begin{lemma}\label{lem:SBLp gamma} 
Let Assumptions \ref{ass:linear A transference} and \ref{ass: B linear} hold. 
Then for every $\tau\in(0,T]$ and $v\in L^p(0,\tau,\wk;X_1)$, one has   $B_iv \in  L^{r_i}(0,\tau,w_{\nu_i};X_{\gamma_i})$ for $1\leq i\leq k$, so $Bv \in\x(0,\tau)$. Moreover, 
\[
\|Bv\|_{\x(0,\tau)}
\leq b_{\scriptscriptstyle[0,\tau]} \|v\|_{L^p(0,\tau,\wk;X_1)}, 
\] 
 where $b_{\scriptscriptstyle[0,\tau]}$ was defined in  \eqref{eq: def b_T}. 
\end{lemma}
\begin{proof}  By Assumption \ref{ass: B linear}, we have  $\|B_i(\cdot)\|_{\calL(X_1,X_{\gamma_i})}\leq b_i(\cdot)$. The claim thus follows from the weighted H\"older inequality \eqref{eq:weightedHolder}.  
\end{proof}

The above lemmas allow us to prove well-posedness of the problem \eqref{eq:ABfu_0}, along with a useful energy estimate. We obtain uniqueness in the space $\MR_{\rm tr}(0,\tau)$ defined in \eqref{eq:MRtr def}, which is larger than $\MR_+(0,\tau)$ due to Corollary \ref{cor: emb}.

\begin{proposition} \label{prop: linearized well-posed}
Let Assumptions \ref{ass:linear A transference} and \ref{ass: B linear} hold.  
Then for any $\tau\in (0,T]$, $u_0\in (X_0, X_1)_{1-\frac{1+\kappa}{p},p}$,  $f\in  \x(0,\tau)$ and $v\in L^p(0,\tau,\wk;X_1)$, there exists a unique strong solution $u\in \MR_+(0,\tau)$ to
  \begin{align}\label{eq:ABfu_0}
\begin{cases}
u' + A u + B v= f ,\\
u(0)=u_0 
\end{cases}
\end{align}
on $[0,\tau]$. Furthermore, the following estimate holds: 
\begin{align*} 
 &\|u\|_{\MR_+(0,\tau)}    
+\|u\|_{L^p(0,\tau,\wk;X_1)}+\|u\|_{C([0,\tau];\Xtr)}+\|u\|_{C_{\nicefrac{\kappa}{p}}((0,\tau];\Xtrzero)}\\
 &  \qquad  \leq 
R_T\|u_0\|_{X_{1-\frac{1+\kappa}{p},p}}     +L_T\big(  b_{\scriptscriptstyle[0,\tau]}\|v\|_{L^p(0,\tau,\wk;X_{1})} +  \|f\|_{\x(0,\tau)}\big), 
\end{align*} 
where $R_T$ is the constant from Lemma \ref{lem:initial mixed scale}, and $L_T$ is a constant depending only on $p,\kappa$, and  $r_i,\nu_i,\gamma_i$ for $1\leq i\leq k$, and nondecreasingly on $M_T$ defined by \eqref{eq: def M_T}, and $C_A$ and  $C_{A_{r_i,\nu_i,\gamma_i}}$  from \eqref{eq:A mapping sum}. In particular, $R_T$ and $L_T$ are independent of $\tau,  u_0, f$, and $v$. 

Uniqueness also holds amongst strong solutions in the larger space
\begin{equation}\label{eq:MRtr def}
\MR_{\rm tr}(0,\tau)\ceqq L^p(0,\tau,w_{\kappa};X_1)\cap C([0,\tau];X_{1-\frac{1+\kappa}{p},p})\cap C_{\nicefrac{\kappa}{p}}((0,\tau];X_{1-\frac{1}{p},p}).
\end{equation}
\end{proposition}
\begin{proof}
Let $\tau\in (0,T]$. 
First, by Lemma \ref{lem:initial mixed scale}, the problem 
$\tilde{u}'+A\tilde{u}=0, 
\tilde{u}(0)=u_0$ has a unique strong solution $\tilde{u}\in L^p(0,\tau,\wk;X_1)\cap W^{1,p}(0,\tau,\wk;X_0)$, and   
\begin{equation}\label{eq:Lp Su_0}
\|\tilde{u}\|_{L^p(0,\tau,\wk;X_1)\cap W^{1,p}(0,\tau,\wk;X_0)}+\|\tilde u\|_{C([0,\tau];\Xtr)} +\|\tilde u\|_{C_{\nicefrac{\kappa}{p}}((0,\tau];\Xtrzero)}\leq  R_T\|u_0\|_{X_{1-\frac{1+\kappa}{p},p}}, 
\end{equation}
where $R_T$ is a constant that depends only on $p,\kappa$, and nondecreasingly on $M_{p,\kappa,A}(0,T)$ and $C_A$. 

Next, let $f = f_0+ \sum_{i=1}^k f_i$ with $f_0\in L^p(0,\tau,\wk;X_0)$ and $f_i\in L^{r_i}(0,\tau,w_{\nu_i};X_{\gamma_i})$. Put $g_0=f_0$ and $g_i=f_i-B_iv$ for $1\leq i\leq  k$. For $i=0$, set $(r_0,\nu_0,\gamma_0)=(p,\kappa,0)$. By Lemma \ref{lem:SBLp gamma}, we have $g_i\in L^{r_i}(0,\tau,w_{\nu_i};X_{\gamma_i})$  for $1\leq i\leq  k$. 
By Lemma \ref{lem:Sf Lr MR transference},  there exists a   strong solution $u^i\in \MR_+(0,\tau)$ to 
\begin{equation}\label{eq: u_i}
\begin{cases}
(u^i)'+Au^i=g_i,\\
u^i(0)=0.
\end{cases}
\end{equation} 
It follows that $u\ceqq \tilde{u}+\sum_{i=0}^k u^i\in\MR_+(0,\tau)$ is a strong solution to \eqref{eq:ABfu_0}.

For uniqueness of strong solutions belonging to $\MR_{\rm tr}(0,\tau)$, suppose that $u_1,u_2\in \MR_{\rm tr}(0,\tau)$ are strong solutions to \eqref{eq:ABfu_0}. Then $w\ceqq u_1-u_2$ satisfies \eqref{eq: in value part} with $u_0=0$. 
In particular, $w'=-Aw$ a.e.\ on $[0,\tau]$ (see \cite[Lem.\ 2.5.8]{HNVWvolume1}). By \eqref{eq:A mapping sum} applied on $[0,\tau]$, we see that $Aw\in L^p(0,\tau,\wk;X_0)$, and thus $w\in L^p(0,\tau,\wk;X_1)\cap W^{1,p}(0,\tau,\wk;X_0)$. Now the maximal $L^p(0,T,\wk;X_0)$-regularity of $A$, specifically \eqref{eq:MRest}, implies that $w=0$, proving uniqueness. 

It remains to prove the energy estimate. Let $f\in \x(0,\tau)$ and 
let $f = f_0+ \sum_{i=1}^k f_i$ be an arbitrary decomposition of $f$  with $f_i\in L^{r_i}(0,\tau,w_{\nu_i};X_{\gamma_i})$. 
We can define $\tilde{u}$ and $u^i$ from Lemma \ref{lem:Sf Lr MR transference} as above (now with potentially different $f_i$) and obtain that $\tilde{u}+\sum_{i=0}^k u^i$ is a strong solution to \eqref{eq:ABfu_0} on $[0,\tau]$ that belongs to $\MR_+(0,\tau)$, 
thus uniqueness gives $u= \tilde{u}+\sum_{i=0}^k u^i$. 
Moreover, by \eqref{eq:Lp Su_0} and by Lemmas \ref{lem:Sf Lr MR transference} and \ref{lem:SBLp gamma}, putting $C_*\ceqq C_A+\max_{1\leq i\leq k}C_{A_{r_i,\nu_i,\gamma_i}}$, we have  
\begin{align}
 \|u\|_{\MR_+(0,\tau)} 
&\leq \|\tilde{u}\|_{\MR_+(0,\tau)} + \textstyle{\sum_{i=0}^k}\|u^i\|_{\MR_+(0,\tau)} \notag\\
&\leq R_T\|u_0\|_{X_{1-\frac{1+\kappa}{p},p}}+
M_T\|f_0\|_{L^p(0,\tau,\wk;X_0)}+ \textstyle{\sum_{i=1}^k}M_T(1+M_T LC_*)\|g_i\|_{L^{r_i}(0,\tau,\wni;X_{\gamma_i})}
 \notag\\ 
& \leq R_T\|u_0\|_{X_{1-\frac{1+\kappa}{p},p}} + M_T \|f_0\|_{L^p(0,\tau,\wk;X_0)} \notag\\
&\qquad\qquad  +M_T(1+M_T LC_*)\Big(b_{\scriptscriptstyle[0,\tau]}\| v\|_{L^p(0,\tau,\wk;X_{1})} +  \textstyle{\sum_{i=1}^k}\|f_i\|_{L^{r_i}(0,\tau,w_{\nu_i};X_{\gamma_i})}\Big), \label{eq:est by split}
\end{align}
where we recall \eqref{eq: def b_T} for the last line. 
Taking the infimum over all decompositions $f_0+\sum_{i=1}^k f_i$ of $f$, we obtain the $\MR_+$-estimate. 

Finally, we prove the remaining estimates. For $i=0,\ldots, k$,   Corollary \ref{cor: emb} and   $u^i(0)=0$ imply  
$$\|u^i\|_{L^p(0,\tau,\wk;X_1)}+\|u^i\|_{C([0,\tau];\Xtr)}+\|u^i\|_{C_{\nicefrac{\kappa}{p}}((0,\tau];\Xtrzero)}\leq K\|u^i\|_{\MR_+(0,\tau)}$$ 
for a  constant $K$ depending only on $p,\kappa,$ and $r_i,\nu_i,\gamma_i$ for $1\leq i\leq k$. 
Moreover, \eqref{eq:Lp Su_0} already provided the necessary estimates for $\tilde{u}$. Combining these estimates with those in \eqref{eq:est by split} completes the proof. 
\end{proof}

\begin{remark}\label{rem:reg}
  The final uniqueness statement of Proposition \ref{prop: linearized well-posed} implies the following regularization: if $\tilde u\in  \MR_{\rm tr}(0,\tau)$ and $\tilde u$ is a strong solution to the nonlinear equation \eqref{eq:parabolic_problem_intro}, then $\tilde u\in \MR_+(0,\tau)$. This follows from applying Proposition \ref{prop: linearized well-posed} with $v=\tilde u$. Indeed, noting that $\tilde u$ solves \eqref{eq:ABfu_0}, uniqueness implies that $\tilde u=u\in \MR_+(0,\tau)$.
\end{remark}

\begin{remark}
  In Theorem \ref{th: well-posed linear}, the arbitrary but fixed admissible triples $(r_i,\nu_i,\gamma_i)$ for $i=1,\ldots,k$ are the same for $f$ and $B$, but this is just for ease of notation. Indeed, $f_i$ or $B_i$ may be chosen as zero, so their admissible triples can just be united with each other.
\end{remark}

After these preparations, we are ready to prove the main result, Theorem \ref{th: well-posed linear}.  The proof follows a structure  similar to that of Theorem \ref{thm:perturbLp}, with the following  main differences: 
\begin{itemize}
\item The solution now belongs to   $\MR_+(0,T)$, which is larger than $\MR^p(0,T)$. In particular, we no longer have $u', Bu\in L^{p}(0,T,w_{\kappa};X_0)$.    
\item The perturbation $Bu$ can already be considered if $u\in L^p(0,T,w_{\kappa};X_1)$.  
 The fixed points of $\Phi$ and $\Phi_n$ are constructed in $L^p(0,T,w_{\kappa};X_1)$. 
\end{itemize}

\begin{proof}[Proof of Theorem \ref{th: well-posed linear}] 

{\em    Step 1 (local solution): } 

For  $\tau\in (0,T]$,  we define a map $\Phi\col L^p(0,\tau,\wk;X_1)\to L^p(0,\tau,\wk;X_1)$ by $\Phi(v)\ceqq u$, where $u\in \MR_+(0,\tau)\subset L^p(0,\tau,\wk;X_1)$ is the strong solution to \eqref{eq:ABfu_0} on $(0,\tau)$, which exists uniquely by Proposition \ref{prop: linearized well-posed}. 
Let $v_1,v_2\in L^p(0,\tau,\wk;X_{1})$. Then $u \ceqq \Phi(v_1) - \Phi(v_2)$ solves
\begin{align*}
\begin{cases}
u' + A u + B (v_1-v_2)= 0 \quad\text{ on }(0,\tau),\\
u(0)=0.
\end{cases}
\end{align*}
Therefore, Proposition \ref{prop: linearized well-posed} gives    
\begin{align}\label{eq: difference est}
\|u\|_{L^p(0,\tau,\wk;X_{1})}\leq  C_0 b_{\scriptscriptstyle[0,\tau]}\|v_1-v_2\|_{L^p(0,\tau,\wk;X_{1})}, 
\end{align}
where we let $C_0\ceqq R_T\vee L_T$, with $R_T$ and $L_T$ the constants from Proposition \ref{prop: linearized well-posed}, and we use that  $\tau\in(0, T]$. 
Note that for \eqref{eq: difference est}, the constant $L_T$ suffices, but we  define $C_0$ as the maximum for later use.    

Choose $\tau\in (0,T]$ such that $b_{\scriptscriptstyle[0,\tau]} = 1/(2C_0)$ (see \eqref{eq: def b_T}) if possible, and otherwise set $\tau = T$. 
It follows from the above that
\begin{align}\label{eq:L contraction}
\|\Phi(v_1) - \Phi(v_2)\|_{L^p(0,\tau,\wk;X_{1})} \leq \frac12 \|v_1- v_2\|_{L^p(0,\tau,\wk;X_{1})}.
\end{align}
By the Banach fixed-point theorem, there exists a unique $u\in L^p(0,\tau,\wk;X_1)$ such that  $\Phi(u) = u$. Recall that $\mathrm{Ran}(\Phi)\subset \MR_+(0,\tau)$ by definition of $\Phi$. Thus, within $\MR_+(0,\tau)$,  $u$ is the unique strong solution to \eqref{eq:parabolic_problem_intro} on $[0,\tau]$.  

{\em  Step 2 ($L^p$-estimate for local solution): } 

For $u$ from Step 1, we have 
\begin{align*}
\| {u}\|_{L^p(0,\tau,\wk;X_1)}
 & \leq \|\Phi( {u}) \|_{L^p(0,\tau,\wk;X_1)}\\
  & \leq \|\Phi( {u}) - \Phi(0)\|_{L^p(0,\tau,\wk;X_1)} + \|\Phi(0) \|_{L^p(0,\tau,\wk;X_1)}
\\ & \leq \frac12 \| {u}\|_{L^p(0,\tau,\wk;X_1)} + \|\Phi(0)\|_{L^p(0,\tau,\wk;X_1)}. 
\end{align*}
Combined with Proposition \ref{prop: linearized well-posed} applied to  $v=0$, we obtain 
\begin{align}\label{eq:aprioriLest} 
\| {u}\|_{L^p(0,\tau,\wk;X_1)}\leq 2\|\Phi(0) \|_{L^p(0,\tau,\wk;X_1)} \leq 2C_0\|u_0\|_{X_{1-\frac{1+\kappa}{p},p}} + 2C_0 \|f \|_{\x(0,\tau)}.
\end{align} 

{\em   Step 3 (iterative construction):} 

We iteratively extend the solution $u$ from Step 1 to intervals $[\tau_{n},\tau_{n+1}]$. Although the power weights for $u$, $f$, and $b_i$ behave like the unweighted case ($w_0=1$) on these intervals, we cannot simply apply Step 1 with $\kappa=0$ (to a shifted  equation), because our final constants should only depend on $T$ and $b_{\scriptscriptstyle[0,T]}$ and not on the lengths of the subintervals.

Instead, we use the following recursive procedure. Set $\tau_0\ceqq 0$ and $\tau_1 \ceqq \tau$ (from Step 1). Let ${C}_0$ be the same constant as in Step 1. 
Then, for $n\geq 2$, choose $\tau_n\in(\tau_{n-1},T)$  such that 
\begin{equation}\label{eq: b pieces}
\textstyle{\sum_{i=1}^k}  \|b_i\|_{L^{\frac{pr_i}{p-r_i}}(\tau_{n-1},\tau_{n},w_{\frac{\nu_i p-\kappa r_i}{p-r_i}})} = 1/(2{C}_0),
\end{equation}    
recalling that  
$b_i \in   L^{\frac{pr_i}{p-r_i}}(0,T,w_{\frac{\nu_i p-\kappa r_i}{p-r_i}})$ by Assumption \ref{ass: B linear}. 
If this is no longer possible, we set $N = n$ and $\tau_N = T$. 

Now fix $n\in \{1, \ldots, N-1\}$, and suppose that  a unique strong solution $u\in\MR_+(0,\tau_n)$ to \eqref{eq:parabolic_problem_intro}  has been constructed on $[0,\tau_n]$. (For $n=1$, this holds since $u=\Phi(u)\in \MR_+(0,\tau)$. ) 
Define $u_n \ceqq u(\tau_n)$. Note that $u_n\in \Xtrzero$ thanks to 
Corollary \ref{cor: emb}. We will now split the problem on $[\tau_n, \tau_{n+1}]$ into two subproblems. In the first problem, we take care of the initial value problem without inhomogeneity:
\begin{align} 
&\begin{cases}
\tilde{w}' + A \tilde{w}  = 0\;\,  \text{ on }(\tau_n,\tau_{n+1}),\\
\tilde{w}(\tau_n)=u_n,
\end{cases} \label{eq: in val tau_n}
\end{align}
The second problem treats the inhomogeneous part, and it is crucial to formulate it on $[0,\tau_{n+1}]$ in order to be able to apply the weighted form of maximal $L^p$-regularity with constants which do not depend on $\tau_{n+1} - \tau_{n}$:
\begin{align}
\begin{cases}
w' + A w + (B\one_{(\tau_n, \tau_{n+1})})v= f\one_{(\tau_n, \tau_{n+1})} -B\tilde{w}\one_{(\tau_n,  \tau_{n+1})}  \;\,\text{ on }(0,\tau_{n+1}),\\
w(0)=0, 
\end{cases} \label{eq: zero till tau_n}
\end{align} 
where $v\in L^p(0,\tau_{n+1},\wk;X_1)$ is given but arbitrary and $\tilde{w}$ is given by \eqref{eq: in val tau_n}. 

Firstly, \eqref{eq: in val tau_n} has a unique strong solution 
$$\tilde{w}\in L^p(\tau_n,\tau_{n+1},\wk;X_1)\cap W^{1,p}(\tau_n,\tau_{n+1},\wk;X_0),$$ 
due to maximal $L^p(0,T,\wk;X_0)$-regularity of $A$ and Lemma \ref{lem:initial}.   
To obtain an $L^p$-estimate for $\tilde{w}$ that does not depend on $ \tau_n$ and $\tau_{n+1} $, we consider $\tilde{v}\ceqq u\one_{[0,\tau_n)}+\tilde{w}\one_{[\tau_n,\tau_{n+1}]}$. 
Since $u(\tau_n)=\tilde{w}(\tau_n)$, we have $\tilde{v}\in \MR_{\rm tr}(0,\tau_{n+1})$ (defined in \eqref{eq:MRtr def}) and $\tilde{v}$ is a strong solution to   
\begin{align*}
\begin{cases}
\tilde{u}' + A\tilde{u} + (B\one_{(0, \tau_n)})\tilde{v}\one_{(0, \tau_n)}= f\one_{(0, \tau_n)}    \quad\text{ on }(0,\tau_{n+1}),\\
\tilde{u}(0)=u_0,  
\end{cases}  
\end{align*}
where $\tilde{B}\ceqq B\one_{(0, \tau_n)}$ satisfies Assumption \ref{ass: B linear} with $\tilde{b}_i\ceqq b_i\one_{(0, \tau_n)}$. 
Thus Proposition \ref{prop: linearized well-posed} and Remark \ref{rem:reg} yield
\begin{align}
\|\tilde{w}&\|_{L^p(\tau_n,\tau_{n+1},\wk;X_1)} \leq 
\|\tilde{v}\|_{L^p(0,\tau_{n+1},\wk;X_1)}\notag\\
&\qquad\leq C_0\Big(\|u_0\|_{X_{1-\frac{1+\kappa}{p},p}} +  \tilde{b}_{\scriptscriptstyle[0,\tau_{n+1}]}\|\tilde{v}\one_{(0, \tau_n)}\|_{L^p(0,\tau_{n+1},\wk;X_{1})} +  \|f\one_{(0, \tau_n)}\|_{\x(0,\tau_{n+1})}\Big)\notag\\
&\qquad= C_0\Big(\|u_0\|_{X_{1-\frac{1+\kappa}{p},p}} +  {b}_{\scriptscriptstyle[0,\tau_n]}\|u\|_{L^p(0, \tau_n,\wk;X_1)}+  \|f\|_{\x(0,\tau_n)}\Big). \label{eq: Lpest tilde w}
\end{align} 
Here, $C_0$ is the constant introduced in Step 1, which is allowed since $\tau_{n+1}\leq T$.

Now we turn to equation \eqref{eq: zero till tau_n}. Note that $\hat{B}\ceqq B\one_{(\tau_n, \tau_{n+1})}$ satisfies Assumption \ref{ass: B linear} with $T$ replaced by $\tau_{n+1}$, with the same admissible triples as for $B$, and with $b_i$ replaced by 
\[
 \hat{b}_i \ceqq b_i \one_{(\tau_n, \tau_{n+1})}\in   L^{\frac{pr_i}{p-r_i}}(0,\tau_{n+1},w_{\frac{\nu_i p-\kappa r_i}{p-r_i}}).
\]
Moreover, we have $f\one_{(\tau_n,\tau_{n+1})}-B\tilde{w}\one_{(\tau_n,\tau_{n+1})}\in  \x(0,T)$ by Lemma \ref{lem:SBLp gamma}. 
Therefore, for any $v\in L^p(0,\tau_{n+1},\wk;X_1)$, Proposition \ref{prop: linearized well-posed} yields a unique strong solution $w\eqqc \Phi_n(v)$ to \eqref{eq: zero till tau_n} with $w\in \MR_+(0,\tau_{n+1})\subset L^p(0,\tau_{n+1},\wk;X_1)$. Moreover,  by virtue of Proposition \ref{prop: linearized well-posed} and \eqref{eq: b pieces},  the analog of  \eqref{eq:L contraction} holds: 
\begin{align*}
\|\Phi_n(v_1)-\Phi_n(v_2)\|_{L^p(0,\tau_{n+1},\wk;X_{1})}&\leq C_0 \hat{b}_{\scriptscriptstyle[0,\tau_{n+1}]}\|v_1-v_2\|_{L^p(0,\tau_{n+1},\wk;X_{1})}\\
&\leq \frac12\|v_1-v_2\|_{L^p(0,\tau_{n+1},\wk;X_{1})}.
\end{align*}
Thus $\Phi_n$ has a unique fixed point $w \in L^p(0,\tau_{n+1},\wk;X_{1})$, and by the definition of $\Phi_n$,  
\[
w=\Phi_n(w)\in \MR_+(0,\tau_{n+1}). 
\]
Moreover (analogous to \eqref{eq:aprioriLest}, now with $u_0=0$): 
\begin{align}\label{eq: w^n est}
\begin{split} 
\|w\|_{L^p(0, \tau_{n+1},\wk;X_1)}&\leq 2\|\Phi_n(0)\|_{L^p(0,\tau_{n+1},\wk;X_1)} \\
&\leq  2C_0 \|f\one_{(\tau_n, \tau_{n+1})}-B\tilde{w}\one_{(\tau_n,\tau_{n+1})}\|_{\x(0,\tau_{n+1})} \\
&  \leq 2C_0 \|f \|_{\x(\tau_n, \tau_{n+1})}+2C_0\hat{b}_{\scriptscriptstyle[0,\tau_{n+1}]}\| \tilde{w} \|_{L^p(\tau_n,\tau_{n+1},\wk;X_1)} \\
&  \leq 2C_0 \|f \|_{\x(\tau_n, \tau_{n+1})}+ \| \tilde{w} \|_{L^p(\tau_n,\tau_{n+1},\wk;X_1)}. 
\end{split}
\end{align}

Now define 
$$
u^n\ceqq w\one_{[\tau_n,\tau_{n+1}]}+\tilde{w}\one_{[\tau_n,\tau_{n+1}]}\in \MR_+(\tau_n,\tau_{n+1}). 
$$ 
Then,   \eqref{eq: Lpest tilde w} and \eqref{eq: w^n est} yield  
\begin{align}
\|u^n&\|_{L^p(\tau_n,\tau_{n+1},\wk;X_1)}\notag\\
 &\,\leq 2C_0\|f\|_{\x(\tau_n,\tau_{n+1})} +  2C_0\Big(\|u_0\|_{\Xtr}+b_{\scriptscriptstyle[0,\tau_n]}\|u\|_{L^p(0,\tau_n,\wk;X_1)}+\|f\|_{\x(0,\tau_n)}\Big),  \notag \\
 &\qquad\leq 4C_0\|f\|_{\x(0,T)} +  2C_0 \|u_0\|_{\Xtr}+n\|u\|_{L^p(0,\tau_n,\wk;X_1)},  
\label{eq:apriori MR n step}\raisetag{-13pt}
\end{align}
where the last line holds by \eqref{eq: b pieces} and   
the fact that
\begin{align*}
 \|f\|_{\x(s,t)}\leq \|f \|_{\x(0,T)},\qquad 0\leq s< t\leq T, 
 \end{align*}
which is easily verified using the sum space definition \eqref{eq:def sum sp f} and \eqref{eq:def sum norm} and considering decompositions of $f\in \x(0,T)$.

Finally, we  extend $u$ to a mapping on $[0,\tau_{n+1}]$ by defining
\[
u=u^n \text{ on }[\tau_n,\tau_{n+1}]. 
\]
By Corollary \ref{cor: emb} and since $u^n(\tau_n)=u_n=u(\tau_n)$, $u\in \MR_{\rm tr}(0,\tau_{n+1})$. Moreover, $u$  is a strong solution to \eqref{eq:parabolic_problem_intro} on $[0,\tau_{n+1}]$, so Proposition \ref{prop: linearized well-posed} applied with $v=u$ yields   $u\in \MR_+(0,\tau_{n+1})$ (see Remark \ref{rem:reg}).
 
To complete the induction step, it remains to argue why $u$ is the  unique strong solution  within $\MR_+(0,\tau_{n+1})$. Let $u_1,u_2\in \MR_+(0,\tau_{n+1})$ be strong solutions to \eqref{eq:parabolic_problem_intro}. By uniqueness in the hypothesis, we find $u_1=u_2$ on $[0,\tau_n]$. In particular, $u_1(\tau_n)=u_2(\tau_n)$, using Corollary \ref{cor: emb}. 
Next, note that within $ \MR_+(\tau_n,\tau_{n+1})$, strong solutions to   
\begin{align*}
\begin{cases}
u' + A u + Bu= f \quad\text{ on }(\tau_n,\tau_{n+1}),\\
u(\tau_n)=u_n 
\end{cases}
\end{align*}
are unique. Indeed, if  $u_1,u_2\in \MR_+(\tau_n,\tau_{n+1})$ are strong solutions to the above equation, then we can define $w_i\ceqq (u_i -\tilde{w})\one_{[\tau_n,\tau_{n+1}]}$ for $i=1,2$, with $\tilde{w}\in  L^p(\tau_n,\tau_{n+1},\wk;X_1)\cap W^{1,p}(\tau_n,\tau_{n+1},\wk;X_0)$ the unique strong solution to \eqref{eq: in val tau_n}. 
Then $w_i\in \MR_{\rm tr}(0,\tau_{n+1})$ (since $w_i(\tau_n)=0$ and by Corollary \ref{cor: emb}) and $w_i$ solves \eqref{eq: zero till tau_n} with $v=w_i$, so Remark \ref{rem:reg} yields $w_i\in\MR_+(0,\tau_{n+1})$. Moreover, $w_1$ and $w_2$ are fixed points for the map $\Phi_n$ defined above. Since $\Phi_n$ has a unique fixed point, we obtain $u_1=u_2$ on $[\tau_n,\tau_{n+1}]$.  
We conclude that the required uniqueness holds on $[0,\tau_{n+1}]$.

We can now repeat the above procedure for $n+1$, until we reach $\tau_N=T$.  

{\em  Step 4 ($L^p$-estimate):} 

In Step 3, we have produced a strong solution $u\in \MR_+(0,T)$ to \eqref{eq:parabolic_problem_intro} on $[0,T]$. 
Next, we collect $L^p$-estimates for $u$ that we already proved during the iteration procedure.   
Indeed,   by  \eqref{eq:apriori MR n step} (for $n\geq 1$) and by \eqref{eq:aprioriLest} (for $n=0$), we obtain for all $n\in\{0,1,\ldots,N-1\}$: 
\begin{align*} 
 \|u \|_{L^p(0,\tau_{n+1},\wk;X_1)} &\leq  \|u\|_{L^p(0,\tau_{n},\wk;X_1)}+\|u^n\|_{L^p(\tau_n,\tau_{n+1},\wk;X_1)}\\ 
 &\leq (1+n)\|u\|_{L^p(0,\tau_n,\wk;X_1)}+ 4C_0\|f\|_{\x(0,T)} +  2C_0 \|u_0\|_{\Xtr}.
\end{align*}
Consequently, the direct formula for this recursion, and estimating $\sum_{j=1}^{N} \frac{1}{j!}\leq  e-1 \leq 7/4$, 
\begin{align}\label{eq: Lp est sum}
\begin{split}
\|u\|_{L^p(0,T,\wk;X_1)}&\leq 4C_0 \big(\|u_0\|_{\Xtr}+\|f \|_{\x(0,T)}\big)N!\textstyle{ \sum_{j=1}^{N}}\frac{1}{j!} \\
&   \leq7N!C_0\big(\|u_0\|_{\Xtr}+\|f \|_{\x(0,T)}\big). 
\end{split}
\end{align} 

Now we bound $N$ from above. If $p=r_i$ for all $i$, then $B=0$ by Assumption \ref{ass: B linear}  and the theorem is trivial. For $B\neq 0$, putting $Q\ceqq \min\{\frac{p-r_i}{pr_i}:1\leq i\leq k,p\neq r_i\}\in(0,1)$, we have for $N\geq 3$:
$$
b_{[0,T]} 
\geq   \textstyle{ \sum_{1\leq i\leq k,p\neq r_i}  (N-2)^{{Q}-1}\sum_{n=1}^{N-2} }\|b_i\|_{L^{\frac{pr_i}{p-r_i}}(\tau_n,\tau_{n+1},w_{\frac{\nu_i p-\kappa r_i}{p-r_i}})}  \geq (N-2)^{{Q}}\frac{1}{2{C}_0},  
$$ 
using \eqref{eq: b pieces} and \eqref{eq: def b_T}. It follows that $N \leq (2{C}_0b_{\scriptscriptstyle[0,T]})^{1/Q}+2\eqqc M$ (also for $N\leq 2$), and from \eqref{eq: Lp est sum},   
we can conclude that
\begin{equation}\label{eq: Lp est for fixed point}
\|u\|_{L^p(0,T,\wk;X_1)}\leq C \Big(\|u_0\|_{X_{1-\frac{1+\kappa}{p}, p}} + \|f\|_{\x(0,T)}\Big),
\end{equation}
where $C\ceqq 7\Gamma(M+1)C_0$ only depends on $C_{0}$, $b_{\scriptscriptstyle[0,T]}$,  $p$ and $r_1,\ldots,r_k$.   

{\em   Step 5 (regularity): } 

The estimate in the statement of Theorem \ref{th: well-posed linear} follows by combining  \eqref{eq: Lp est for fixed point} from Step 4 with an application of  Proposition \ref{prop: linearized well-posed} with $u=v$. 
\end{proof}

\begin{remark}\label{rem:Amappingchoice}
Regarding the mapping condition \eqref{eq:A mapping sum} in Assumption \ref{ass:linear A transference}, there is some flexibility in choosing the function space for $u$ without invalidating the results of Theorem \ref{th: well-posed linear}. Indeed, in  \eqref{eq:A mapping sum}, 
it suffices to estimate by $C_A\|u\|_{E(0,T)}$ for any normed space $E(0,T)$ such that $\MR_+(0,T)\into E(0,T)\into L^p(0,T,\wk;X_1)$, as  seen from the proof of Proposition \ref{prop: linearized well-posed}. However, choosing   the smallest space $E(0,T)=\MR_+(0,T)$ can lead to a constant $C_A$ that diverges as $T\downarrow 0$, which may be undesirable.  
Our current choice prevents this behavior and covers  virtually all cases of interest. 
\end{remark}

\begin{remark}\label{rem:kanook}
By Remark \ref{rem:reg},  the existence and uniqueness   in Theorem \ref{th: well-posed linear} also hold within the space  $\MR_{\rm tr}(0,T) = L^p(0,T,w_{\kappa};X_1)\cap C([0,T];X_{1-\frac{1+\kappa}{p},p})\cap C_{\nicefrac{\kappa}{p}}((0,T];X_{1-\frac{1}{p},p})$, 
for which we have $\MR_+(0,T)\hookrightarrow \MR_{\rm tr}(0,T)$ by Corollary \ref{cor: emb}.
Moreover, using a transference argument as in Section \ref{sec:borderline}, maximal $L^r$-regularity can be avoided  entirely, as long as there exists an operator  $A_0$ such that $-A_0$ generates an analytic $C_0$-semigroup on $X_0$ and $D(A_0) = X_1$. In that case, one can show that for $f\in L^{r}(0,T,\wn;X_{\gamma})$ with $(p,\kappa)$-admissible $(r,\nu,\gamma)$, 
\[\Big\|t\mapsto \textstyle{\int_0^t} e^{-(t-s)A_0} f(s)  \dd s\Big\|_{\MR_{\rm tr}(0,T)}\lesssim \|f\|_{L^{r}(0,T,\wn;X_{\gamma})}.\]
By the Da Prato--Grisvard theorem (see \cite[Corollary 17.3.20]{HNVWvolume3}), $A_{0}$ has maximal
$L^r$-regularity on $X_{\gamma,\infty}$. By weighted
extrapolation theorem for maximal regularity this gives that the solution is in 
$W^{1,r}(0,T,w_\nu;X_{\gamma,\infty})  \cap L^r(0,T,w_\nu;X_{1+\gamma,\infty})$. It remains to apply Proposition
\ref{prop: emb Lr} and the weighted trace embedding Lemma \ref{lem:trace}.
\end{remark}

\begin{remark}[Generalization to intermediate spaces]\label{rem:generalized_B}
Theorem \ref{th: well-posed linear} can be extended to perturbations $B_i\col [0,T]\to \calL(Y_{\beta_i},X_{\gamma_i})$ acting on intermediate spaces. For $i=1,\dots,k$, assume that $\beta_i \in (1-\frac{1+\kappa}{p}, 1]$ and $Y_{\beta_i}$ is a Banach space satisfying $X_{\beta_i, 1} \hookrightarrow Y_{\beta_i}$.  
We replace the integrability condition in Assumption \ref{ass: B linear} by  $\|B_i(\cdot)\|_{\calL(Y_{\beta_i}, X_{\gamma_i})} \leq b_i(\cdot)$, where $b_i \in L^{q_i}(0,T,w_{\mu_i})$ with parameters 
$$
q_i \ceqq \frac{pr_i}{p-r_i\theta_i}, \quad \mu_i \ceqq \frac{\nu_i p-\kappa r_i\theta_i}{p-r_i\theta_i}, \qquad \text{where }\, \theta_i \ceqq 1-\frac{(1-\beta_i)p}{1+\kappa},\quad p-r_i\theta_i>0.
$$ 
The well-posedness in $\MR_+(0,T)$ and estimates of Theorem \ref{th: well-posed linear} remain valid in this setting. 
The proof follows the same fixed-point strategy but relies on the auxiliary space $E_{p,\kappa}(0,\tau) \ceqq L^p(0,\tau,w_\kappa; X_1) \cap L^\infty(0,\tau; X_{1-\frac{1+\kappa}{p}, p})$. Crucially, using the reiteration theorem and Assumption \ref{ass:linear A transference}\textit{(\ref{it:A1})}, one has $X_{\beta_i, 1} = (X_{1-\frac{1+\kappa}{p}, p}, X_1)_{\theta_i, 1}$, which implies the following interpolation estimate:
$$\|B_i v\|_{L^{r_i}(0,\tau, w_{\nu_i}; X_{\gamma_i})} \leq C \|b_i\|_{L^{q_i}(0,\tau,w_{\mu_i})} \|v\|_{L^\infty(0,\tau; X_{1-\frac{1+\kappa}{p}, p})}^{1-\theta_i} \|v\|_{L^p(0,\tau, \wk; X_1)}^{\theta_i}, \quad v\in E_{p,\kappa}(0,\tau).$$  
This allows one to reprove  Proposition \ref{prop: linearized well-posed} for $v\in E_{p,\kappa}(0,\tau)$ (with $\|v\|_{E_{p,\kappa}(0,\tau)}$ in the estimate) and close the contraction arguments in the proof of Theorem \ref{th: well-posed linear} in the spaces $E_{p,\kappa}(0,\tau)$ (Step 1) and $E_{p,\kappa}(0,\tau_{n+1})$ (Step 3).
\end{remark} 

\subsection{Further embeddings of the mixed-scale maximal regularity space}

We recall that Theorem \ref{th: well-posed linear}  already provided continuity properties of solutions, which were derived from Corollary \ref{cor: emb}. In this subsection, we explore further regularity properties, in particular, fractional smoothness in time. 

We  establish useful embeddings for the space $\MR_+(0,T)$ (see \eqref{eq: def MR space}). Note that all embeddings in this subsection rely on an extension operator.
Consequently, the operator norm of each embedding depends on $T$ and tends to infinity as $T\downarrow 0$, unless the domain is restricted to functions in   $\MR_+(0,T)$ vanishing at $t=0$.

For $\delta\in(0,1)$ and a Banach space $X$, we define
\[
H^{\delta,p}(0,T,\wk;X)\ceqq   
  [L^{p}(0,T,\wk;X),W^{1,p}(0,T,\wk;X)]_\delta. 
\]

\begin{proposition}\label{prop: emb Lr Hdeltap}
  Let $p\in (1, \infty)$ and let $\kappa\in[0,p-1)$. 
Let $(r, \nu, \gamma)$ be $(p,\kappa)$-admissible. 
Then  for all $\delta\in[0,1-\gamma)$, it holds that
\begin{equation}\label{eq:emb to show Hdeltap}
\begin{split}
  L^r(0,T&,\wn;X_{1+\gamma})\cap W^{1,r}(0,T,\wn;X_{\gamma}) \into H^{\delta,p}(0,T,\wk;X_{1-\delta,1}). 
  \end{split}
\end{equation} 
\end{proposition}

\begin{proof}
To prove \eqref{eq:emb to show Hdeltap}, we use the embeddings of \eqref{eq:emb1}  for some $s\leq 0$, together with the following additional embeddings:
\begin{align}
&B_{\hat{r}\hat{r}}^{s+1}(\R,w_{\hat{\nu}};X_{\gamma})\cap B_{\hat{r}\hat{r}}^{s}(\R,w_{\hat{\nu}};X_{1+\gamma}) \into B_{\hat{r}\hat{r}}^{{{z}}}(\R,w_{\hat{\nu}};X_{1-\delta,1})
\label{eq:emb2 Hdeltap}\\[.25em]
&B_{\hat{r}\hat{r}}^{{{z}}}(\R,w_{\hat{\nu}};X_{1-\delta,1})\into  H^{\delta,p}(\R,\wk;X_{1-\delta,1})
\label{eq:emb3 Hdeltap}
\end{align}
Establishing these embeddings on $\R$ suffices, as we can apply a bounded extension operator $\mathrm{E}$ and restrict to $(0,T)$ subsequently, as in the proof of Proposition \ref{prop: emb Lr}. So it remains to show that there exist parameters $s,{{z}},\hat{r},\hat{\nu}$  such that the embeddings are valid.

To justify the use of \eqref{eq:emb1}, we impose the parameter conditions \eqref{eq:emb1cond} from the proof of Proposition \ref{prop: emb Lr}. Next, the embeddings of \eqref{eq:emb2 Hdeltap} and \eqref{eq:emb3 Hdeltap} will be proved analogously to those of \eqref{eq:emb2} and \eqref{eq:emb3}.   

For \eqref{eq:emb2 Hdeltap} we apply \cite[Th.\ 3.1]{MV14traces} with  $\tilde{\theta}=1-\gamma-\delta\in(0,1)$ 
and spaces  
$\tilde{X}_0=X_{\gamma}$, $\tilde{X}_1=X_{1+\gamma}$, $\tilde{X}_{\tilde{\theta}}=X_{1-\delta,1}$, 
and with all other parameters as indicated above \eqref{eq:emb2cond}.  These parameters satisfy the conditions of \cite[Th.\ 3.1]{MV14traces} if and only if 
 \begin{equation}\label{eq:emb2cond2}
s+\gamma+\delta={{z}}, 
\end{equation}  
see also the proof of Proposition \ref{prop: emb Lr}, and note that  the current choice of $\tilde{X}_{\tilde{\theta}}$ also satisfies \cite[(3.1)]{MV14traces}, since by \eqref{eq:Xgamma squeeze} and reiteration (see \cite[Th.\ L.3.1]{HNVWvolume3}):
\[
(X_{\gamma},X_{1+\gamma})_{\tilde{\theta},1}\into (X_{\gamma,\infty},X_{1+\gamma,\infty})_{\tilde{\theta},1}= X_{1-\delta,1}.
\] 

Finally, we turn to \eqref{eq:emb3 Hdeltap}. We use that $B_{rr}^s=F_{rr}^s$ and apply \cite[Prop.\ 3.12,  (3.8)]{MV12sharp} (for $F_{p,1}^\delta\into H^{\delta,p}$) and \cite[Th.\ 1.2]{MV12sharp}, imposing that  
 \begin{equation}\label{eq:emb3cond2}
p\geq \hat{r}, \quad
\frac{\kappa}{p}\leq \frac{\hat{\nu}}{\hat{r}},\quad \frac{1+\kappa}{p}<\frac{1+\hat{\nu}}{\hat{r}}, 
\quad {{z}}-\frac{1+\hat{\nu}}{\hat{r}}\geq \delta-\frac{1+\kappa}{p}.
\end{equation} 

Now,   conditions \eqref{eq:emb1cond}, \eqref{eq:emb2cond2} and \eqref{eq:emb3cond2} can be rewritten as (see also \eqref{eq:embconds}):
\begin{equation*} 
\begin{cases}
&p\geq \hat{r}\geq r, \\
& \frac{\kappa}{p}\leq \frac{\hat{\nu}}{\hat{r}}\leq \frac{\nu}{r}, \\
&\frac{1+\kappa}{p}<\frac{1+\hat{\nu}}{\hat{r}}<\frac{1+\nu}{r},  \\
&s\leq \frac{1+\hat{\nu}}{\hat{r}}-\frac{1+\nu}{r}, \\ 
& s\geq-\gamma+\frac{1+\hat{\nu}}{\hat{r}}  -\frac{1+\kappa}{p},\\
& s+\gamma+\delta={{z}}.
\end{cases}
\end{equation*}
By the reasoning at the end of the proof of Proposition \ref{prop: emb Lr}, we conclude that 
there exist $s,{{z}},\hat{r},\hat{\nu}$ satisfying the   conditions above if and only if 
\eqref{eq:adm for emb} holds, which corresponds to $(p,\kappa)$-admissibility of $(r,\nu,\gamma)$.  
\end{proof}

\begin{remark}\label{rem:Hdeltap alternative sp}
We note that $X_{1-\delta,1}$ embeds continuously into $X_{1-\delta}$ (by \eqref{eq:Xgamma squeeze}),  into $[X_0,X_{1+\gamma_*}]_{(1-\delta)/(1+\gamma_*)}$ (by \eqref{eq:def Xgammaq} and  Peetre's theorem \cite[Th.\ C.4.1]{HNVWvolume1}) and into $[X_0,X_1]_{1-\delta}$. In particular, in the right-hand side of the embedding \eqref{eq:emb to show Hdeltap}, we can replace $X_{1-\delta,1}$ by any of these spaces. 

For the embedding $X_{1-\delta,1}\into [X_0,X_1]_{1-\delta}$, if $\delta\in(0,1)$, we use subsequently \eqref{eq:def Xgammaq}, reiteration \cite[Th.\ L.3.1]{HNVWvolume3},  \eqref{eq:def Xgammaq}, \eqref{eq:Xgamma squeeze} and Peetre's theorem \cite[Th.\ C.4.1]{HNVWvolume1}:
   \begin{align*}
   X_{1-\delta,1}=(X_0,X_{1+\gamma_*})_{\frac{1-\delta}{1+\gamma_*},1}&=(X_0,(X_0,X_{1+\gamma_*})_{\frac{1}{{1+\gamma_*}},1})_{{1-\delta},1}
   \\&=(X_0,X_{1,1})_{1-\delta,1}\into (X_0,X_{1})_{1-\delta,1}\into [X_0,X_{1}]_{1-\delta},
   \end{align*} 
and if $\delta=0$, we use Assumption \ref{ass:linear A transference}\textit{(\ref{it:A1})}.
\end{remark}

Next, in what follows, we will be assuming that $X_0$ and $X_1$ are UMD spaces (see \cite[\S4.2]{HNVWvolume1}). 
Under the UMD assumption, we have thanks to \cite[Th.\ 3.18]{LV20} and \cite[Lem.\ 5.4]{LMV18} (for the domain $(0,T)$),  for all $\delta\in[0,1]$:
\begin{equation}\label{eq:Hdp interpol}
  H^{\delta,p}(0,T,\wk;[X_0,X_1]_{1-\delta})\cong   [L^{p}(0,T,\wk;X_1),W^{1,p}(0,T,\wk;X_0)]_\delta. 
\end{equation} 

\begin{corollary} \label{cor: emb UMD} 
Let $p\in (1, \infty)$ and let $\kappa\in[0,p-1)$. Suppose that $X_0$ and $X_1$ are UMD spaces. 
Let $(r_i, \nu_i, \gamma_i)$ be $(p,\kappa)$-admissible for $1\leq i \leq k$. 
Define $\delta_*=\min\{1-\gamma_i:1\leq i\leq k\}$. 

For all $\delta\in[0,\delta_*)$, it holds that
\begin{align*} 
  \MR_+(0,T)&\into  H^{\delta,p}(0,T,\wk;[X_0,X_1]_{1-\delta}).
  \end{align*}
\end{corollary} 
\begin{proof} 
Observe that thanks to \eqref{eq:Hdp interpol}, for all $\delta\in[0,1]$, 
\begin{equation}\label{eq: interpol identity UMD}
  L^p(0,T,\wk; X_1)\cap W^{1,p}(0,T,\wk;X_0)\into  H^{\delta,p}(0,T,\wk;[X_0,X_1]_{1-\delta}). 
\end{equation}  
  Recalling \eqref{eq: def MR space}, Proposition \ref{prop: emb Lr Hdeltap} and Remark \ref{rem:Hdeltap alternative sp} conclude the proof.  
\end{proof}

\begin{remark}\label{rem:complex emb UMD} 
If  $X_\gamma$ is chosen as $[X_0,X_{1+\gamma_*}]_{\frac{\gamma}{1+\gamma_*}}$, then we can also include the borderline case $\delta=\delta_*$ in Corollary \ref{cor: emb UMD}. 
  
Note that \eqref{eq: interpol identity UMD} already held for all $\delta\in[0,1]$, and $X_\gamma=[X_0,X_1]_{\gamma}$ for $\gamma\in[0,1]$ by reiteration,  so it suffices to show  that   $L^r(0,T,w_{\nu};X_{1+\gamma})\cap W^{1,r}(0,T,w_{\nu};X_{\gamma}) \into H^{1-\gamma,p}(0,T,\wk;X_{\gamma})$, whenever $(r,\nu,\gamma)$ is $(p,\kappa)$-admissible. 
  Now, by \cite[Th.\ 1.2, Prop.\ 3.12]{MV12sharp}, the stronger embedding $W^{1,r}(0,T,w_{\nu};X_{\gamma}) \into H^{1-\gamma,p}(0,T,\wk;X_{\gamma})$ even holds, provided that $r\leq p$, $\frac{\kappa}{p}\leq \frac{\nu}{r}$, $ \frac{1+\kappa}{p}<\frac{1+\nu}{r}$ and $\gamma-\frac{1+\nu}{r}\geq 0-\frac{1+\kappa}{p}$, which are  satisfied thanks to the assumed $(p,\kappa)$-admissibility. 
\end{remark}

\begin{remark}
If $X_0$ and $X_{1+\gamma_*}$ in Assumption \ref{ass:linear A transference} are UMD spaces  and   $X_\gamma= [X_0,X_{1+\gamma_*}]_{\gamma/(1+\gamma_*)}$, then $X_\gamma$ and $X_{1+\gamma}$ are also UMD, and Proposition \ref{prop: emb Lr} admits a shorter proof via \eqref{eq:Hdp interpol}:
\begin{align*}
  L^r(0,T,w_{\nu};X_{1+\gamma})\cap W^{1,r}(0,T,w_{\nu};X_{\gamma}) &\into [L^r(0,T,w_{\nu};X_{1+\gamma}), W^{1,r}(0,T,w_{\nu};X_{\gamma})]_\gamma\\
  &=H^{\gamma,r}(0,T,\wn;X_{1})\\
  &\into L^p(0,T,\wk;X_{1}),
\end{align*}
where the last embedding follows from \cite[Th.\ 1.2, Prop.\ 3.12]{MV12sharp}, using $(p,\kappa)$-admissibility, and we use  reiteration for complex interpolation \cite[\S1.10.3 Th.\ 2]{Tri78} in the second line. 
\end{remark}

\section{Perturbations beyond maximal $L^p$-regularity: the  case $r=1$}\label{sec:borderline}

In this section, we present a version of Theorem \ref{th: well-posed linear} for the limiting case $r=1$. The latter necessitates setting $\nu=0$ (to ensure integrability of $f$), hence $\kappa=0$, due to the admissibility condition $\kappa/p\leq \nu/r$. 
Standard maximal $L^1$-regularity is often an unsuitable tool, notably failing on reflexive Banach spaces (see \cite[Cor.\ 17.4.5]{HNVWvolume3}). Consequently, the maximal $L^r$-regularity requirement in Assumption \ref{ass:linear A transference} becomes problematic, rendering  Section \ref{sec:linear adm} inapplicable. 

Instead, we establish estimates directly, using a modified solution space to replace $\MR_+(s,t)$:
\begin{align}
E_p(s,t)& \ceqq L^p(s,t;X_1)\cap C([s,t];X_{1-\nicefrac1p,p}), \label{eq:DefEp}
\\ F_p(s,t)&\ceqq L^p(s,t;X_0) + L^1(s,t;(X_0,X_1)_{1-\nicefrac1p,p}).\label{eq:DefFp}
\end{align}

In the current section we treat $p\in (1, \infty)$ and set $\gamma = 1-\nicefrac{1}{p}$ (cf.\ Definition \ref{def:admissible}). Furthermore, the perturbation $B(t)$ maps from $X_1$ to $(X_0,X_1)_{1-\nicefrac{1}{p},p}$. 

\subsection{Main result}

The following assumption will be sufficient for our main result. 

\begin{assumption}\label{assum:HScase}
Assumption \ref{assump:XA} holds and $p\in (1, \infty)$. Moreover: 
\begin{enumerate}[\rm (1)]
\item $A\col[0,T]\to \calL(X_1, X_0)$ has maximal $L^p(0,T;X_0)$-regularity and
\begin{equation}\label{eq:A mapping r=1}
  \|Au\|_{L^p(0,T;X_0)}\leq C_A\|u\|_{E_p(0,T)}, \qquad u\in E_p(0,T)  
\end{equation} 
for a constant $C_A$. 
\item\label{it2:HScase} There exists a closed operator $A_0$ on $X_0$ such that $-A_0$ generates an analytic $C_0$-semigroup on $X_0$ and $D(A_0) = X_1$.  
\item\label{it3:HScase}   $B\col [0,T]\to \calL(X_1,(X_0,X_1)_{1-\nicefrac1p,p})$ is   strongly measurable in the strong operator topology, and  
$\|B(\cdot)\|_{\calL(X_1,(X_0,X_1)_{1-\nicefrac1p,p})}\leq b(\cdot)$ for some $b\in L^{p'}(0,T)$. 
\end{enumerate}
\end{assumption}
As below Assumption \ref{ass:linear A transference} one can check that \eqref{eq:A mapping r=1} extends to all subintervals. 

As in Section \ref{sec:linear adm}, it would be possible to consider $\sum_{i=1}^k B_i$, where each  $B_i$ satisfies the condition of Assumption \ref{ass: B linear} or Assumption \ref{assum:HScase}\eqref{it3:HScase} (see Remark \ref{rem:kanook}).

When $X_0$ is a Hilbert space, Assumption \ref{assum:HScase} simplifies considerably:
\begin{remark}\label{rem:nonvar}
Suppose that $X_0$ is a Hilbert space. Then the following assertions hold: 
\begin{enumerate}
\item If $A$ is time-independent and $X_1=D(A)$, the assumption on $A$ is equivalent to $-A$ being the generator of an analytic $C_0$-semigroup by De Simon's theorem. (see \cite[Cor.\ 17.3.8]{HNVWvolume3}). In particular, we can take $A_0 = A$. 
\item Assumption \ref{assum:HScase}(\ref{it2:HScase}) holds automatically if $X_1$ is also a Hilbert space. 
    Indeed, the space
$X_0$ can be identified with $X_1^*$ using the duality pairing $\lb\cdot,\cdot\rb$ between $X_0$ and $X_1$  induced by $[X_0,X_1]_{1/2}$ \cite[\S5.5.2]{arendt02}. Consequently, we can define $A_0\in\mathcal{L}(X_1,X_0)$ by  $\lb A_0x,y\rb\ceqq (x,y)_{X_1}$, which in particular satisfies $\lb A_0 x,x\rb=\|x\|_{X_1}^2$.  Then $A_0$ is invertible, positive,  and self-adjoint. In particular, $-A_0$ generates an analytic $C_0$-semigroup. 
\item If $p=2$, then the real interpolation space $(X_0,X_1)_{1-\nicefrac1p,p}$ coincides with the complex one (see \cite[Corollary C.4.2]{HNVWvolume1}). Therefore, $E_2(s,t) = L^2(s,t;X_1)\cap C([s,t];[X_0, X_{1}]_{\nicefrac12})$.
\end{enumerate}
\end{remark}

Our main perturbation result for $r=1$ can now be stated as follows. 
\begin{theorem}\label{thm:mainHilbert}
Suppose that Assumption \ref{assum:HScase} holds and let $E_p(0,T)$ and $F_p(0,T)$ be defined as in \eqref{eq:DefEp}  and  \eqref{eq:DefFp}. Then for all $u_0\in (X_0,X_1)_{1-\nicefrac1p,p}$  and $f\in F_p(0,T)$,  the problem 
\begin{equation*}
\left\{
\begin{aligned}
    u' + (A+B)u &= f, \\
    u(0) &= u_0
\end{aligned}
\right.
\end{equation*}
 has a unique strong solution $u\in E_p(0,T)$ and
\begin{align*}
\|u\|_{E_p(0,T)}  \leq C \big(\|u_0\|_{(X_0,X_1)_{1-\nicefrac1p,p}} +  \|f\|_{F_p(0,T)}\big)   
\end{align*} 
where  $C$ is a   constant independent of $u_0$ and $f$, which depends only on $A_0,X_0,X_1,p$,   and nondecreasingly on $\|b\|_{L^{p'}(0,T)},T,C_A$, and $M_{p,0,A}(0,T)$.
\end{theorem} 

Even in the unperturbed case ($B=0$), it is noteworthy that the space of inhomogeneities can be enlarged to $F_p(0,T)$ while preserving the regularity $u\in L^p(0,T;X_1)$. 
Also, in case $X_0$ is a Hilbert space and $p=2$, Theorem \ref{thm:mainHilbert} extends \cite[Cor.\ 3.5]{TV24} to the nonvariational setting (see Remark \ref{rem:nonvar}). 

\subsection{Proof of Theorem \ref{thm:mainHilbert}}
We start with a general semigroup lemma. We thank Emiel Lorist for suggesting the elementary proof provided below.
\begin{lemma}\label{lem:L2real} 
Let $X_0$ and $X_1$ be Banach spaces and let $p\in (1, \infty)$. Let $A_0$ be a closed operator on $X_0$ such that $-A_0$ generates an analytic $C_0$-semigroup with $X_1 = D(A_0)$. Then 
for every $g\in L^1(0,T;X_{1-\nicefrac1p,p})$ and $u_0\in (X_0, X_1)_{1-\nicefrac1p,p}$ the problem
\begin{equation*}
\left\{
\begin{aligned}
    u' + A_0 u &= g, \\
    u(0) &= u_0.
\end{aligned}
\right.
\end{equation*}
has a unique strong solution $u\in E_p(0,T)$ and
\begin{align*}
\|u\|_{E_p(0,T)}\leq C\|u_0\|_{X_{1-\nicefrac1p,p}} + C \|g\|_{L^1(0,T;X_{1-\nicefrac1p,p})},
\end{align*}
where $C$ is a constant independent of $u_0$ and $g$, and depends only on $A_0,X_0,X_1,p$,   and nondecreasingly on $T$.  
\end{lemma}
\begin{proof}
Uniqueness follows from the fact that weak solutions are unique \cite[Th.\ G.3.2]{HNVWvolume2}. 

Before we proceed with the existence and estimate, observe that by 
applying the invertible transformation $u\mapsto e^{-\lambda\cdot} u$, we have an equivalence of the problem above to the same problem with $A_0$  replaced by $\lambda+A_0$.
Therefore, without loss of generality, we may assume that $A_0$  is sectorial of angle $<\pi/2$ and invertible. In particular, $e^{-tA_0}$ is exponentially stable. 

Extend $g$ by zero to a function on $(0,\infty)$. 
By \cite[Th.\ G.3.2]{HNVWvolume2}, there exists a weak solution $u$ which can be represented as a mild solution:
\begin{align*}
u(t) = \exp(-tA_0)u_0 +  \int_0^t \exp(-(t-s)A_0) g(s) ds \eqqc \exp(-tA_0)u_0 + v(t), \ \ t\geq 0.
\end{align*}
It suffices to prove the bounds for $T=\infty$. 

The $C([0,\infty);X_{1-\nicefrac1p,p})$-regularity and the corresponding estimate follow from the fact that $t\mapsto e^{-t A_0}$ restricts to an exponentially stable analytic $C_0$-semigroup on $X_{1-\nicefrac1p,p}$. 
It is also standard that $\|t\mapsto \exp(-tA_0)u_0\|_{L^p(\R_+;X_1)}\leq C\|u_0\|_{X_{1-\nicefrac1p,p}}$ (see \cite[Th.\ L.2.4]{HNVWvolume3}).

Next, we prove $L^p(0,\infty;X_1)$-regularity of the mild solution $u$ given above. We have 
\begin{align*}
\|v\|_{L^p(0,\infty;X_1)} & \eqsim \Big\|t\mapsto \int_0^\infty A_0 \exp(-(t-s)A_0) \one_{(0,t)}(s) g(s) ds \Big\|_{L^p(0,\infty;X_0)}
\\& \lesssim \int_0^\infty \|t\mapsto A_0 \exp(-(t-s)A_0) \one_{(0,t)}(s) g(s)\|_{L^p(0,\infty;X_0)} ds 
\\& = \int_0^\infty \|t\mapsto A_0 \exp(-t A_0)  g(s)\|_{L^p(0,\infty;X_0)} ds 
\\& \eqsim \|g\|_{L^1(0,\infty;X_{1-\nicefrac{1}{p},p})}, 
\end{align*}
where in the last step we used \cite[Th.\ L.2.4]{HNVWvolume3}.  

Finally, note that by \cite[Prop.\ 17.1.3]{HNVWvolume3} and the previously obtained regularity, the mild solution $u$ is a strong solution as well. 
\end{proof}

Next, we extend this bound to the non-autonomous case. 

\begin{proposition}\label{prop:HS}
 Let Assumption \ref{assum:HScase} be satisfied.  
Then for every $\tau\in(0,T]$, $u_0\in (X_0,X_1)_{1-\nicefrac1p,p}$, $f\in F_p(0,\tau)$ and $v\in L^p(0,\tau;X_1)$, the problem 
\begin{equation*}
\left\{
\begin{aligned}
    u' + A u +Bv &= f, \\
    u(0) &= u_0
\end{aligned}
\right.
\end{equation*}
has a unique strong solution $u\in E_p(0,\tau)$ on $[0,\tau]$, and 
\[\|u\|_{E_p(0,\tau)} \leq C_0 \Big(\|u_0\|_{(X_0,X_1)_{1-\nicefrac1p,p}}+\|b\|_{L^{p'}(0,\tau)}\|v\|_{L^p(0,\tau;X_1)}+\|f\|_{F_p(0,\tau)}\Big),
\]
where  $C_0$ is a constant independent of $\tau, u_0, f$, and $v$, which  depends only on $A_0,X_0,X_1,p$,  and nondecreasingly on $T,C_A$, and $M_{p,0,A}(0,T)$.    
\end{proposition}

\begin{proof} 
For uniqueness, let $w\ceqq u_2-u_1$ be the difference of two strong solutions in $E_p(0,\tau)$. 
Then $w'=-Aw$ a.e.\ (see \cite[Lem.\ 2.5.8]{HNVWvolume1}).  
We conclude that $w'=-Aw\in L^p(0,\tau;X_0)$. Since \eqref{eq:A mapping r=1} extends to all subintervals, it follows that $w \in L^p(0,\tau;X_1)\cap W^{1,p}(0,\tau;X_0)$ and maximal $L^p(0,\tau;X_0)$-regularity from Lemma \ref{lem:restriction} yields $w=0$. 

To prove existence, we employ the transference method of \cite[Th.\ 3.13]{AV25survey}. 
Let $A_0$ be as in Assumption \ref{assum:HScase}. Let $f = g+h$ with $g\in L^1(0,\tau;(X_0, X_{1})_{1-\nicefrac1p,p})$ and $h\in L^p(0,\tau;X_0)$. By Lemma \ref{lem:L2real},  there exists a unique strong solution $\tilde{v}\in E_p(0,T)$  to $\tilde{v}'+A_0 \tilde{v} =-\one_{[0,\tau]}Bv+ \one_{[0,\tau]} g$,  $\tilde{v}(0) = u_0$. Moreover,   
\[\|\tilde{v}\|_{E_p(0,\tau)} \leq \|\tilde{v}\|_{E_p(0,T)}\leq C\big(\|u_0\|_{(X_0, X_{1})_{1-\nicefrac1p,p}}+
 \|b\|_{L^{p'}(0,\tau)}\|v\|_{L^p(0,\tau;X_1)}+ \|g\|_{L^1(0,\tau;(X_0, X_{1})_{1-\nicefrac1p,p})}\big),\] 
using that Assumption \ref{assum:HScase}(\ref{it3:HScase}) and H\"older's inequality yield 
\begin{equation}\label{eq: Bv est r=1}
  \|Bv\|_{L^1(0,\tau;(X_0, X_{1})_{1-\nicefrac1p,p})}\leq \|b\|_{L^{p'}(0,\tau)}\|v\|_{L^p(0,\tau;X_1)}. 
\end{equation}
By \eqref{eq:A mapping r=1}, maximal $L^p(0,T;X_0)$-regularity of $A$ and Lemmas \ref{lem:restriction} and \ref{lem:initial}, we can find a unique strong solution $z\in L^p(0,T;X_1)\cap W^{1,p}(0,T;X_0)$ to $z' + A z = (A_0-A)\tilde{v}+\one_{[0,\tau]} h$, $z(0) = 0$,  
and combined with the trace embedding \eqref{eq:traceembedding0T}  and \eqref{eq:A mapping r=1} applied to $\tilde{v}$, 
\begin{align*}
\|z\|_{E_p(0,\tau)}&\leq \|z\|_{E_p(0,T)} 
\\& \lesssim \|z\|_{L^p(0,T;X_1)\cap W^{1,p}(0,T;X_0)}
\\ & \leq M(\|A_0\|_{\calL(X_1,X_0)}+C_A) \|\tilde{v}\|_{E_p(0,T)}+M\|\one_{[0,\tau]}h\|_{L^p(0,T;X_0)}, 
\end{align*}
where $M\ceqq M_{p,0,A}(0,T)$, and the   constant in the second line  is $\tau$- and $T$-independent since $z(0)=0$. 
One can verify that $u \ceqq \tilde{v}+z\in E_p(0,\tau)$ is the required strong solution.
Moreover, $\|u\|_{E_p(0,\tau)}\leq \|\tilde{v}\|_{E_p(0,\tau)} + \|z\|_{E_p(0,\tau)}$, so the claimed estimate follows from the estimates  for $\tilde{v}$ and $z$ above, and by taking the infimum over all $g$ and $h$ as above. 
\end{proof}

By adapting the proof of Theorem \ref{th: well-posed linear}, we obtain the following result.
\begin{proof}[First proof of Theorem \ref{thm:mainHilbert}]
  We can repeat the proof of Theorem \ref{th: well-posed linear}, replacing $\MR_+(s,t)$ by $E_p(s,t)$ and $\x(s,t)$ by $F_p(s,t)$. As a fixed-point space, we use $L^p(s,t;X_1)$ again. 
  Note that \eqref{eq: Bv est r=1} replaces Lemma \ref{lem:SBLp gamma} and Proposition \ref{prop:HS} replaces Proposition \ref{prop: linearized well-posed}. 
  
  With the substitutions above, we have $\mathrm{Ran}(\Phi)\subset E_p(0,\tau)$ in Step 1 and $\mathrm{Ran}(\Phi_n)\subset E_p(0,\tau_{n+1})$ in Step 3, thus we can exploit the uniqueness within $E_p(0,\tau)$ (instead of $\MR_+(0,\tau)$) from Proposition \ref{prop:HS}. Moreover, we note that $E_p(0,t)$ has the property that $v_1\one_{[0,s]}+v_2\one_{(s,t]}\in E_p(0,t)$ when $v_1\in E_p(0,s)$, $v_2\in E_p(s,t)$ and $v_1(s)=v_2(s)$. This property, previously used in the proof of Theorem \ref{th: well-posed linear}, also holds for $E_p(s,t)$. 
\end{proof}
 Alternatively, the proof can be simplified due to the absence of weights in this setting.
  
  \begin{proof}[Second proof of Theorem \ref{thm:mainHilbert}]
  In Step 1, define $\Phi$ on $L^p(0,\tau;X_1)$. Note that \eqref{eq: difference est}--\eqref{eq:aprioriLest} remain valid. 
  
  In Step 3, again let $\tau_{n}$ be such that $\|b\|_{L^{p'}(\tau_{n-1},\tau_{n})} = 1/(2{C}_0)$ if possible, and otherwise set $\tau_n = T$. In the unweighted setting, we can iterate Step 1 using the shifted coefficients $\tilde{A}\ceqq A(\cdot+\tau_n)$, $\tilde{B}\ceqq B(\cdot+\tau_n)$, and $\tilde{f}\ceqq f(\cdot+\tau_n)$. Recall that \eqref{eq:A mapping r=1} is also satisfied on subintervals. Consider on $[0,\tau_{n+1}-\tau_n]$:  
  \[
  \left\{
\begin{aligned}
    \tilde{u}' + \tilde{A}\tilde{u}+\tilde{B}v &= \tilde{f}, \\
    \tilde{u}(0) &= u(\tau_n).
\end{aligned}
\right.
  \]  
Define $\Phi_n$ on $L^p(0,\tau_{n+1}-\tau_n;X_1)$ as $\Phi_n(v) = \tilde{u}$, where $\tilde{u}\in E_p(0,\tau_{n+1}-\tau_n)$ is the unique strong solution to the equation above, which exists by  Lemmas \ref{lem:restriction} and \ref{lem:initial} and Proposition \ref{prop:HS} (with $\tilde{b}\ceqq b(\cdot+\tau_n)$).   A unique fixed point $u^n=\Phi_n(u^n)$ is guaranteed since 
\begin{align*}
\|\Phi_n(v_1)-\Phi_n(v_2)\|_{E_p(0,\tau_{n+1}-\tau_n)}\leq C_0 \|b\|_{L^{p'}(\tau_n,\tau_{n+1})}\|v_1-v_2\|_{L^p(0,\tau_{n+1}-\tau_n;X_{1})}, 
\end{align*}
which holds with the same constant $C_0$ as in Step 1, thanks to Lemmas \ref{lem:restriction} and \ref{lem:initial}. The solution $u$ can  be extended by defining $u=u^n(\cdot-\tau_n)$ on $(\tau_n,\tau_{n+1}]$. 

Moreover, similarly to \eqref{eq:aprioriLest}, one derives
\begin{align*}
\|u\|_{E_p(0,\tau_{n+1})}&\leq \|u\|_{E_p(0,\tau_{n})}+
\|u^n\|_{E_p(0,\tau_{n+1}-\tau_n)} \\
&\leq \|u\|_{E_p(0,\tau_{n})}+2C_0\|u(\tau_n)\|_{(X_0, X_1)_{1-\nicefrac1p,p}}+2C_0\|f\|_{F_p(\tau_n,\tau_{n+1})}, \\
&\leq (1+2C_0)\|u\|_{E_p(0,\tau_n)}+2C_0\|f\|_{F_p(0,T)}. 
\end{align*}
This recursive estimate, together with \eqref{eq:aprioriLest} (now with $E_p(0,\tau)$-norm) yields an estimate ($\tau_N=T$)
\[
\|u\|_{E_p(0,T)} \leq 2C_0(1+2C_0)^{N-1}\|u_0\|_{(X_0,X_1)_{1-\nicefrac1p,p}} +\big[(1+2C_0)^N - 1\big] \|f\|_{F_p(0,T)},
\]
and $N$ can be bounded in terms of $\|b\|_{L^{p'}(0,T)}$ as was done at the end of Step 4.
\end{proof}

\end{document}